\documentclass[11pt]{article}
\usepackage[T1]{fontenc}
\usepackage{amsfonts}
\usepackage{amsmath}
\usepackage{amssymb}
\usepackage{amsthm}
\usepackage{bbm}
\usepackage{bm}
\usepackage{mathrsfs}
\usepackage{verbatim}
\usepackage{setspace}
\usepackage{color}
\usepackage{pdfsync}
\usepackage{enumitem}
\usepackage{graphicx}

\theoremstyle{plain}
\newtheorem{theorem}{Theorem}[section]
\newtheorem{proposition}[theorem]{Proposition}
\newtheorem{lemma}[theorem]{Lemma}
\newtheorem{corollary}[theorem]{Corollary}

\theoremstyle{definition}
\newtheorem{definition}[theorem]{Definition}
\newtheorem{remark}[theorem]{Remark}

\newtheorem{example}[theorem]{Example}

\theoremstyle{remark}

\renewenvironment{thebibliography}[1]{%
\begin{oldthebibliography}{#1}%
\setlength{\baselineskip}{.9em}
\linespread{1}
\small
\setlength{\parskip}{0.3ex}%
\setlength{\itemsep}{.5em}%
}%
{%
\end{oldthebibliography}%
}
\newcommand{\eps}{\varepsilon}

\newcommand{\N}{\mathbb{N}}

\newcommand{\R}{\mathbb{R}}

\newcommand{\cD}{\mathcal{D}}

\newcommand{\cF}{\mathcal{F}}

\newcommand{\cI}{\mathcal{I}}

\newcommand{\cM}{\mathcal{M}}

\newcommand{\cP}{\mathcal{P}}

\newcommand{\cU}{\mathcal{U}}

\newcommand{\bC}{\mathbf{C}}

\newcommand{\bI}{\mathbf{I}}

\newcommand{\bS}{\mathbf{S}}

\DeclareMathOperator{\USA}{USA}

\DeclareMathOperator{\supp}{supp}

\DeclareMathOperator{\Var}{Var}
\DeclareMathOperator{\conv}{conv}

\newcommand{\as}{\mbox{-a.s.}}
\newcommand{\qs}{\mbox{-q.s.}}
\newcommand{\1}{\mathbf{1}}

\numberwithin{equation}{section}

\usepackage[pdfborder={0 0 0}]{hyperref}
\hypersetup{
  urlcolor = black,
  pdfauthor = {Mathias Beiglbock, Marcel Nutz, Nizar Touzi},
  pdfkeywords = {Martingale Optimal Transport; Kantorovich Duality},
  pdftitle = {Complete Duality for Martingale Optimal Transport on the Line},
  pdfsubject = {Complete Duality for Martingale Optimal Transport on the Line},
  pdfpagemode = UseNone
}

\begin{document}

\title{\vspace{-0em}
Complete Duality for\\Martingale Optimal Transport on the Line
\date{\today}
\author{
  Mathias Beiglb\"ock%
  \thanks{University of Vienna, Department of Mathematics, mathias.beiglboeck@univie.ac.at. Financial support by FWF Grants P26736 and Y782-N25 is gratefully acknowledged.}
  \and
  Marcel Nutz%
  \thanks{
  Columbia University, Departments of Statistics and Mathematics, mnutz@columbia.edu. Financial support by NSF Grants DMS-1512900 and DMS-1208985 is gratefully acknowledged.
  }
  \and
  Nizar Touzi%
  \thanks{Ecole Polytechnique Paris, CMAP, nizar.touzi@polytechnique.edu. This work benefits from the financial support of the ERC Advanced Grant 321111, the ANR grant ISOTACE, and the Chairs \emph{Financial Risk} and \emph{Finance and Sustainable Development}.\newline \indent The authors thank Florian Stebegg and an anonymous referee for helpful comments.
  }
 }
}
\maketitle \vspace{-1.2em}

\begin{abstract}
 We study the optimal transport between two probability measures on the real line, where the transport plans are laws of one-step martingales. A quasi-sure formulation of the dual problem is introduced and shown to yield a complete duality theory for general marginals and measurable reward (cost) functions: absence of a duality gap and existence of dual optimizers. Both properties are shown to fail in the classical formulation. As a consequence of the duality result, we obtain a general principle of cyclical monotonicity describing the geometry of optimal transports.

\end{abstract}

\vspace{0.9em}

{\small

}

\section{Introduction}\label{se:intro}

Let $\mu,\nu$ be probability measures on the real line $\R$. A Monge--Kantorovich transport from $\mu$ to $\nu$ is a probability $P$ on $\R^{2}$ whose marginals are $\mu$ and~$\nu$, respectively; that is, if $(X,Y)$ is the identity map on $\R^{2}$, then $\mu=P\circ X^{-1}$ is the distribution of $X$ under $P$ and similarly $\nu=P\circ Y^{-1}$. The set of all these transports is denoted by $\Pi(\mu,\nu)$. Let $P\in\Pi(\mu,\nu)$ and consider the disintegration $P=\mu\otimes \kappa$. If the stochastic kernel $\kappa(x,dy)\equiv P[\,\cdot\,|X=x]$ is given by the Dirac mass $\delta_{T(x)}$ for a map $T:\R\to\R$, then $T$ is called the corresponding Monge transport. In general, a Monge--Kantorovich transport may be interpreted as a randomized Monge transport. 

Let $f$ be a (measurable) real  function on $\R^{2}$; then the cumulative reward for transporting $\mu$ to $\nu$ according to $P$ is
$$
  P(f) \equiv E^{P}[f(X,Y)] \equiv \int_{\R^{2}} f(x,y) \,P(dx,dy)
$$
and the Monge--Kantorovich optimal transport problem is given by
\begin{equation}\label{eq:introOT}
  \sup_{P\in\Pi(\mu,\nu)} P(f).
\end{equation}
In an alternate interpretation, the negative of $f$ is seen as a cost and the above is the minimization of the cumulative cost. One advantage of the Monge--Kantorovich formulation is that an optimizer $P\in\Pi(\mu,\nu)$ exists as soon as $f$ is upper semicontinuous and sufficiently integrable (of course, existence may fail when $f$ is merely measurable).
Optimal transport has been a very active field in the last several decades; we refer to Villani's monographs~\cite{Villani.03,Villani.09} or the lecture notes by Ambrosio and Gigli~\cite{AmbrosioGigli.13} for background.

In the so-called \emph{martingale optimal transport} problem, we only consider transports which are martingale laws; then $\mu$ can be seen as the distribution of a martingale at time $t=0$ and $\nu$ as the distribution of the process at $t=1$. This problem was introduced by Beiglb\"ock, Henry-Labord\`ere and Penkner~\cite{BeiglbockHenryLaborderePenkner.11} in the discrete-time case and by Galichon, Henry-Labord\`ere and Touzi~\cite{GalichonHenryLabordereTouzi.11} in continuous time. In the present paper, we focus on the most fundamental case, where the transport takes place in a single time step. That is, a martingale transport from $\mu$ to $\nu$ is a law $P\in \Pi(\mu,\nu)$ under which $(X,Y)$ is a martingale; of course, this necessitates that $\mu$ and $\nu$ have finite first moments. We let 
$$
  \cM(\mu,\nu) = \big\{P\in \Pi(\mu,\nu):\, E^{P}[Y|X]=X\; P\as \big\}
$$
denote the set of martingale transports. Alternately, consider a disintegration $P=\mu\otimes \kappa$ of $P\in \Pi(\mu,\nu)$; then $P$ is a martingale transport if and only if~$x$ is the barycenter (mean) of $\kappa(x)$ for $\mu$-a.e.\ $x\in\R$; that is, $\int y \,\kappa(x,dy)=x$. Here we may also observe that Monge transports are meaningless in this context---only a constant martingale is deterministic.

The martingale property induces an asymmetry between $\mu$ and $\nu$---the marginals can only become more dispersed over time. More precisely, the set $\cM(\mu,\nu)$ is nonempty if and only if $\mu,\nu$ are in convex order, denoted $\mu\leq_{c}\nu$, meaning that $\mu(\phi)\leq\nu(\phi)$ whenever $\phi$ is a convex function (see Proposition~\ref{pr:convexOrder}). Under this condition, the martingale optimal transport problem is given by
\begin{equation}\label{eq:introMOT}
  \sup_{P\in\cM(\mu,\nu)} P(f).
\end{equation}
The present paper develops a complete duality theory for this problem, for general reward functions and marginals. In particular, we obtain existence in the dual problem, and that is the main goal of this paper.

The problem~\eqref{eq:introMOT} was first studied in~\cite{BeiglbockJuillet.12,HobsonNeuberger.12}. In analogy to the Hoeffding--Fr\'echet coupling of classical transport, \cite{BeiglbockJuillet.12} establishes a measure~$P$, the so--called Left-Curtain Coupling, that is optimal in~\eqref{eq:introMOT} for reward functions~$f$ of a specific form.  This form was generalized to a version of the Spence--Mirrlees condition in \cite{HenryLabordereTouzi.13}, where the coupling is also described more explicitly, whereas~\cite{Juillet.14} shows the stability with respect to the marginals. On the other hand, \cite{HobsonKlimmek.15,HobsonNeuberger.12} find the optimal transports for $f(x,y)=\pm |x-y|$. A generalization of the martingale transport problem, where an arbitrary linear constraint is imposed on $\Pi(\mu,\nu)$, is studied in \cite{Zaev.14}.

Martingale optimal transport is motivated by considerations of model uncertainty in financial mathematics. Starting with~\cite{Hobson.98}, a stream of literature studies robust bounds for option prices via the Skorokhod embedding problem and this can be interpreted as optimal transport in continuous time; cf.\ \cite{Hobson.11, Obloj.04} for surveys. The opposite direction is taken in \cite{BeiglbockCoxHuesman.14}, where Skorokhod embeddings are studied from an optimal transport point of view. Recently, a rich literature has emerged around the topics of model robustness and transport; see, e.g., \cite{AcciaioBeiglbockPenknerSchachermayer.12, BeiglbockNutz.14, BouchardNutz.13, BurzoniFrittelliMaggis.15, CampiLaachirMartini.14, CheriditoKupperTangpi.14, FahimHuang.14, Nutz.13} for models in discrete time and \cite{BeiglbockHenryLabordereTouzi.15, BiaginiBouchardKardarasNutz.14, CoxObloj.11, CoxHouObloj.14, DolinskySoner.12, DolinskySoner.14, HenryLabordereOblojSpoidaTouzi.12, HenryLabordereTanTouzi.14, KallbladTanTouzi.15, NeufeldNutz.12, Nutz.14, Stebegg.14, TanTouzi.11} for continuous-time models.

\subsection{Duality for Classical Transport}

Let us first recall the duality results for the classical case~\eqref{eq:introOT}. Indeed, the dual problem is given by
\begin{equation}\label{eq:introOTdual}
  \inf_{\varphi,\psi} \, \{\mu(\varphi)+\nu(\psi)\},\quad\mbox{subject to}\quad  \varphi(x)+\psi(y)\geq f(x,y), \quad (x,y)\in\R^{2}.
\end{equation}
Here $\varphi\in L^{1}(\mu)$ and $\psi\in L^{1}(\nu)$ are real functions that can be seen as Lagrange multipliers for the marginal constraints in~\eqref{eq:introOT}. There are two fundamental results on this duality in a general setting, obtained by Kellerer~\cite{Kellerer.84}. First, there is \emph{no duality gap}; i.e., the values of~\eqref{eq:introOT} and~\eqref{eq:introOTdual} coincide. Second, \emph{there exists an optimizer} $(\varphi,\psi)\in L^{1}(\mu)\times L^{1}(\nu)$ for the dual problem, whenever its value~\eqref{eq:introOTdual} is finite. While additional regularity assumptions allow for easier proofs, the results of~\cite{Kellerer.84} apply to \emph{any} Borel function $f: \R^{2}\to [0,\infty]$.  An important application is the ``Fundamental Theorem of Optimal Transport'' \cite{AmbrosioGigli.13, Villani.09} or ``Monotonicity Principle'' which describes the trajectories used by optimal transports: there exists a set $\Gamma\subseteq \R^{2}$ such that a given transport $P\in\Pi(\mu,\nu)$ is optimal for~\eqref{eq:introOT} if and only if $P$ is concentrated on~$\Gamma$. This set can be obtained directly from a dual optimizer $(\varphi,\psi)$ by setting
\begin{equation}\label{eq:introGamma}
    \Gamma = \{(x,y)\in\R^{2}:\, \varphi(x)+\psi(y)=f(x,y)\}.
\end{equation}
In fact, given $\psi$, one can find $\varphi$ by $f$-concave conjugation and vice versa, so that either of the functions may be called the Kantorovich potential of the problem, and then $\Gamma$ is the graph of its $f$-subdifferential. The set $\Gamma$ has an important property called $f$-cyclical monotonicity which can be used to analyze the geometry of optimal transports; we refer to~\cite{AmbrosioGigli.13, Villani.09} for further background.

\subsection{Duality for Martingale Transport}

Let us now move on to the dual problem in the case of interest, where the martingale constraint gives rise to an additional Lagrange multiplier. Formally, $E^{P}[Y|X]=X$ is equivalent to $E^{P}[h(X)(Y-X)]=0$ for all functions $h$ and thus the domain of the analogue of~\eqref{eq:introOTdual} consists of triplets $(\varphi,\psi,h)$ of real functions such that
\begin{equation}\label{eq:introMOTpw}
  \varphi(x)+\psi(y)+h(x)(y-x)\geq f(x,y), \quad (x,y)\in\R^{2},
\end{equation}
while the dual cost function is unchanged,
\begin{equation*}%
  \inf_{\varphi,\psi,h}\, \{\mu(\varphi)+\nu(\psi)\}.
\end{equation*}
In~\cite{BeiglbockHenryLaborderePenkner.11}, it was shown that there is no duality gap whenever the reward function~$f$ is \emph{upper semicontinuous} and satisfies a linear growth condition, and the analogous result holds in the setting of~\cite{Zaev.14}. On the other hand, a counterexample in~\cite{BeiglbockHenryLaborderePenkner.11} showed that the dual problem may \emph{fail} to admit an optimizer, even if $f$ is bounded continuous and the marginals are compactly supported.

The proofs of the positive results in~\cite{BeiglbockHenryLaborderePenkner.11,Zaev.14}, absence of a duality gap, reduce to classical transport theory by dualizing the martingale constraint and using a minimax argument. Only the latter step requires upper semicontinuity, and it is easy to believe that it is a technical condition necessitated only by the technique of proof. This turns out to be wrong: we provide a counterexample (Example~\ref{ex:dualityGap}) showing that the dual problem~\eqref{eq:introMOTpw} can produce a duality gap in a fairly tame setting with compactly supported marginals and a reward function that is bounded and lower semicontinuous. Regarding the absence of an optimizer, we provide a counterexample (Example~\ref{ex:noAttainment}) which is, in a sense to be made specific, simpler than the one in~\cite{BeiglbockHenryLaborderePenkner.11} and suggests that failure of existence is generic as soon as the marginals do not satisfy a condition called irreducibility (see below and Section~\ref{se:prelimConvexOrder}) and $f$ is not smooth.

Let us now introduce a formulation of the dual problem which will allow us to overcome both issues and develop a complete duality theory---dual existence and no duality gap---for general reward functions.
The most important novelty is that we shall reformulate the pointwise inequality of~\eqref{eq:introMOTpw} in a quasi-sure way. Indeed, we say that a property holds \emph{$\cM(\mu,\nu)$-quasi-surely} (q.s.\ for short) if it holds outside a $\cM(\mu,\nu)$-polar set; that is, a set which is $P$-null for all $P\in\cM(\mu,\nu)$. 
We then replace~\eqref{eq:introMOTpw} by
\begin{equation}\label{eq:introMOTqs}
  \varphi(X)+\psi(Y)+h(X)(Y-X)\geq f(X,Y) \quad \cM(\mu,\nu)\qs;
\end{equation}
i.e., the inequality holds $P$-a.s.\ for all $P\in\cM(\mu,\nu)$. For the classical transport, it is known that all polar sets are of a trivial type---they are negligible for one of the marginals. This is different in the martingale case. Indeed, as observed in~\cite{BeiglbockJuillet.12}, there are obstacles that cannot be crossed by any martingale transport. These barriers divide the real line into intervals that (almost) do not interact and are therefore called irreducible components. Our first important result (Theorem~\ref{th:polarDescr}) provides a complete characterization of the $\cM(\mu,\nu)$-polar sets: a subset of $\R^{2}$ is polar if and only if it consists of trajectories a) crossing a barrier or b) negligible for one of the marginals. On the strength of this result, we have a rather precise understanding of~\eqref{eq:introMOTqs}; namely, it represents a pointwise inequality on each irreducible component, modulo sets that are not seen by the marginals.

We thus proceed by first studying an irreducible component; the analysis has two parts. On the one hand, there are soft arguments of separation (Hahn--Banach) and extension (Choquet theory) that are familiar from classical transport theory. On the other hand, there is an important closedness result (Proposition~\ref{pr:closednessIrred}) based on novel arguments: given reward functions $f_{n}\to f$ and corresponding almost-optimal dual elements $(\varphi_{n},\psi_{n},h_{n})$, we construct a limit $(\varphi,\psi,h)$ for~$f$.  The proof of this result is deeply linked to the convex order of the marginals. Indeed, we introduce concave functions~$\chi_{n}$ which control $(\varphi_{n},\psi_{n})$ in the sense of one-sided bounds. 
A compactness result of Arzela--Ascoli type is established for the sequence $(\chi_{n})$, based on a bound of the form 
\begin{equation}\label{eq:introCompactness}
  0\leq \int \chi_{n} \, d(\mu-\nu)\leq C.
\end{equation}
After finding a limit $\chi$ for $\chi_{n}$, we can produce limits $(\varphi,\psi)$ for $(\varphi_{n},\psi_{n})$ by Komlos-type arguments, and the corresponding function $h$ can be found in an a posteriori fashion. The compactness result yields some insight into the failure of the pointwise formulation~\eqref{eq:introMOTpw} for the global problem: the bound~\eqref{eq:introCompactness} does not control the concavity of~$\chi_{n}$ at barriers because the inequality between $\mu$ and $\nu$ is not ``strict'' in the convex order  at these points. 

A second relaxation is necessary to obtain our duality result; namely, the cost $\mu(\varphi)+\nu(\psi)$ needs to be defined in an extended sense. We provide counterexamples showing that the existence of dual optimizers (Example~\ref{ex:noIntegrability}) and in some cases the absence of a duality gap (Example~\ref{ex:gapWithIntegrability}) break down if one insists on $\varphi$ and $\psi$ being integrable for $\mu$ and $\nu$, individually. We shall see that several natural definitions of $\mu(\varphi)+\nu(\psi)$ lead to the same value.

With these notions in place, our main result (Theorem~\ref{th:dualityGlobal}) is that duality holds for arbitrary Borel reward functions $f:\R^{2}\to [0,\infty]$; here the lower bound can be relaxed easily (Remark~\ref{rk:lowerBound}) but not eliminated completely (Example~\ref{ex:dualityGapLower}). Moreover, existence holds in the dual problem whenever it is finite. As a  consequence, we derive a monotonicity principle (Corollary~\ref{co:monotonicityPrinciple}) with a set analogous to~\eqref{eq:introGamma} in a fairly definitive form, generalizing and simplifying results of \cite{BeiglbockJuillet.12, Zaev.14}.

While there are no previous results on duality for irregular reward functions, we mention that the proof of the monotonicity principle in \cite{BeiglbockJuillet.12} contains elements of a theory for dual optimizers for the case of continuous $f$, although the dual problem as such is not formalized in \cite{BeiglbockJuillet.12}. We expect that the quasi-sure formulation proposed in the present paper will prove to be a useful framework not only for the situation at hand but for a large class of transport problems; in particular, to obtain dual attainment under general conditions.

The remainder of the paper is organized as follows. In Section~\ref{se:prelimConvexOrder}, we recall preliminaries on the convex order and potential functions. The structure of $\cM(\mu,\nu)$-polar sets is characterized in Section~\ref{se:polarSets}, and Section~\ref{se:generalizedIntegral} discusses the extended definition of $\mu(\varphi)+\nu(\psi)$. The crucial closedness result for the dual problem is obtained in Section~\ref{se:closednessOnIrred}, which allows us to establish the duality on an irreducible component in Section~\ref{se:dualityOnIrred}. Section~\ref{se:mainResults} combines the previous results to obtain the global duality theorem and the monotonicity principle. The counterexamples are collected in the concluding Section~\ref{se:counterexamples}.

\section{Preliminaries on the Convex Order}\label{se:prelimConvexOrder}

It will be useful to consider finite measures, not necessarily normalized to be probabilities. The notions introduced in Section~\ref{se:intro} extend in an obvious way. 
Let $\mu,\nu$ be finite measures on $\R$ with finite first moment. We say that~$\mu$ and~$\nu$ are in \emph{convex order}, denoted $\mu\leq_{c}\nu$, if $\mu(\phi)\leq \nu(\phi)$ for any convex function $\phi: \R\to \R$. It then follows that $\mu$ and $\nu$ have the same total mass and the same barycenter. An alternative characterization of this order refers to the so-called potential function, defined by
$$
  u_{\mu}: \R\to\R,\quad   u_{\mu}(x) := \int |t-x|\, \mu(dt).
$$
This is a nonnegative convex function with a minimum at the median of~$\mu$, and $\mu$ can be recovered from $u_{\mu}$ via the second derivative measure.
The following result is known; the nontrivial part is \cite[Theorem~8]{Strassen.65}.

\begin{proposition}\label{pr:convexOrder} Suppose that $\mu(\R)=\nu(\R)$. The following are equivalent:
  \begin{enumerate}
  \item The measures $\mu$ and $\nu$ are in convex order: $\mu\leq_{c}\nu$.
  \item The potential functions of $\mu$ and $\nu$ are ordered: $u_{\mu}\leq u_{\nu}$.
  \item There exists a martingale transport from $\mu$ to $\nu$: $\cM(\mu,\nu)\neq \emptyset$.  \end{enumerate}
\end{proposition}

It will be important to distinguish the intervals where $u_{\mu}< u_{\nu}$ from the points where the potential functions touch, because such points act as barriers for martingale transports. In all that follows, the statement $\mu\leq_{c}\nu$ implicitly means that $\mu,\nu$ are finite measures on $\R$ with finite first moment. %

\begin{definition}\label{de:irred}
  The pair $\mu\leq_{c}\nu$ is \emph{irreducible} if the set $I=\{u_{\mu}< u_{\nu}\}$ is connected and $\mu(I)=\mu(\R)$. In this situation, let~$J$ be the union of~$I$ and any endpoints of $I$ that are atoms of $\nu$; then~$(I,J)$ is the \emph{domain} of $(\mu,\nu)$.
\end{definition}

As $u_{\mu}= u_{\nu}$ outside of $I$ and $\mu(I)=\mu(\R)$ and $\mu$, $\nu$ have the same mass and mean, the measure $\nu$ is concentrated on $J$. More precisely, the open interval~$I$ is the interior of the convex hull of the support of~$\nu$, and $J$ is the minimal superset of $I$ supporting~$\nu$. The marginals $\mu\leq_{c}\nu$ can be decomposed into irreducible components as follows; cf.\ \cite[Theorem 8.4]{BeiglbockJuillet.12}.

\begin{proposition}\label{pr:decomp}
  Let $\mu\leq_{c}\nu$ and let $(I_{k})_{1\leq k \leq N}$ be the (open) components of $\{u_{\mu}<u_{\nu}\}$, where $N\in \{0,1,\dots,\infty\}$. Set $I_{0}=\R\setminus \cup_{k\geq1} I_{k}$ and $\mu_{k}=\mu|_{I_{k}}$ for $k\geq 0$, so that $\mu=\sum_{k\geq0} \mu_{k}$.
  Then, there exists a unique decomposition $\nu=\sum_{k\geq0} \nu_{k}$ such that
  $$
    \mu_{0} = \nu_{0} \quad\quad \mbox{and}\quad\quad \mu_{k}\leq_{c} \nu_{k} \quad \mbox{for all} \quad k\geq1,
  $$
  and this decomposition satisfies $I_{k}=\{u_{\mu_{k}}<u_{\nu_{k}}\}$ for all $k\geq1$.
  Moreover, any $P\in\cM(\mu,\nu)$ admits a unique decomposition
  $
    P=\sum_{k\geq0} P_{k}
  $
  such that $P_{k}\in\cM(\mu_{k},\nu_{k})$ for all $k\geq0$.
\end{proposition}

	The index $0$ is special in the above: the measure $P_{0}$ is the unique martingale transport from $\mu_{0}$ to itself, given by the law of $x\mapsto (x,x)$ under $\mu_{0}$. This corresponds to a constant martingale or the identical Monge transport. In particular, $P_{0}$ does not depend on $P\in\cM(\mu,\nu)$. We observe that $P_{0}$ is concentrated on $\Delta_{0} := \Delta \cap I_{0}^{2}$, the part of the diagonal $\Delta=\{(x,x)\in\R^{2}:\, x\in\R\}$ which is not contained in any of the squares $I_{k}\times J_{k}$ for $k\geq 1$. Thus, $\Delta_{0}$ will play a role similar to $I_{k}\times J_{k}$ for $k=0$.

A second remark is that both of the families $(\mu_{k})_{k\geq0}$ and $(P_{k})_{k\geq0}$ are mutually singular, whereas $(\nu_{k})_{k\geq0}$ need not be. Indeed, an atom of $\nu$ may be split such as to contribute to two adjacent components $\nu_{k}$.

We close this section with a technical remark for later use.

\begin{remark}\label{rk:irredAtoms}
   We observe from the definition that the continuous convex function $u_{\mu}$ is affine to the left and to the right of the support of $\mu$, with absolute slope equal to the mass of $\mu$. Moreover, discontinuities of the first derivative correspond to atoms of $\mu$.
   
   Let $\mu\leq_{c}\nu$ be irreducible with domain $(I,J)$ and write $I=(l,r)$. As $\mu(I)=\mu(\R)$, the measure $\mu$ cannot have atoms at the boundary points of~$I$.
  Suppose that $r<\infty$; then the derivative $d u_{\mu}(r)$ exists and is equal to~$\mu(\R)$. However, the measure $\nu$ may have an atom at $r$, and while the right derivative $d^{+} u_{\nu}(r)$ is always equal to $d u_{\mu}(r)$, the left derivative satisfies
  $$
    d u_{\mu}(r) - d^{-} u_{\nu}(r) = 2\nu(\{r\}).
  $$
  Similarly, if $l>-\infty$, we have $d^{-} u_{\nu}(l)=d u_{\mu}(l)=-\mu(\R)$ and 
  $$
    d^{+} u_{\nu}(l) -  d u_{\mu}(l)= 2\nu(\{l\}).
  $$
\end{remark}

\section{The Structure of $\cM(\mu,\nu)$-Polar Sets}\label{se:polarSets}

The goal of this section is to characterize the sets which cannot be charged by any martingale transport. Given a collection $\cP$ of measures on some space~$(\Omega,\cF)$, a set $B\subseteq\Omega$ is called \emph{$\cP$-polar} if is it $P$-null for every $P\in\cP$.

For the classical mass transport, the following result can be obtained by applying Kellerer's duality theorem~\cite{Kellerer.84} to the indicator function $f=\1_{B}$; cf.\ \cite[Proposition~2.1]{BeiglbockGoldsternMareschSchachermayer.09}.

\begin{proposition}\label{pr:polarDescrClassical}
  Let $\mu,\nu$ be finite measures of the same total mass and let $B\subseteq \R^{2}$ be a Borel set. Then $B$ is $\Pi(\mu,\nu)$-polar if and only if there exist a $\mu$-nullset $N_{\mu}$ and a $\nu$-nullset $N_{\nu}$ such that
  $$
     B \subseteq (N_{\mu}\times \R) \cup (\R \times N_{\nu}). 
  $$
\end{proposition}

The above result, which holds true more generally for arbitrary Polish spaces, states that the only $\Pi(\mu,\nu)$-polar sets are the obvious ones: the sets which are not seen by the marginals. This is the reason why in the classical dual transport problem, there is no difference between a quasi-sure formulation and a pointwise formulation. Namely, if $\varphi(X)+\psi(Y)\geq f$ holds $\Pi(\mu,\nu)$-q.s., let $B$ be the exceptional set and let $N_{\mu}$, $N_{\nu}$ be as above. Then setting $\varphi=\infty$ on $N_{\mu}$ and $\psi=\infty$ on $N_{\nu}$ yields $\varphi(X)+\psi(Y)\geq f$ pointwise on~$\R^{2}$, without changing the cost $\mu(\varphi)+\nu(\psi)$.

The situation is fundamentally different for the martingale transport. Unless $\mu\leq_{c}\nu$ is irreducible, there are obstructions to all martingale transports, and more precisely, a set that ``fails to be on a component'' is polar, even if it is seen by the marginals. The following result completely describes the structure of $\cM(\mu,\nu)$-polar sets.

\begin{figure}[h]
  \caption{In this illustration of Theorem~\ref{th:polarDescr}, the striped areas correspond to the domains of two irreducible components. The dotted areas are polar even though they are not negligible for the marginals.} 
  \begin{center}
    \includegraphics[scale=.25]{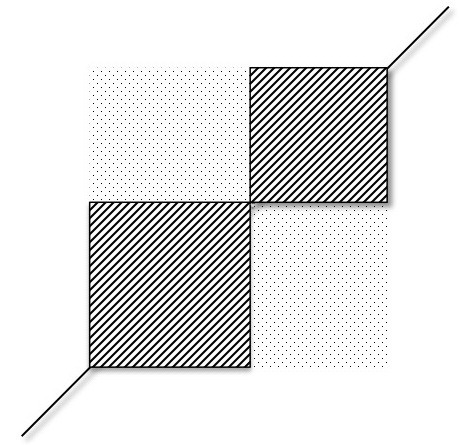}
  \end{center}
\end{figure}

\begin{theorem}\label{th:polarDescr}
  Let $\mu\leq_{c}\nu$ and let $B\subseteq \R^{2}$ be a Borel set. Then $B$ is $\cM(\mu,\nu)$-polar if and only if there exist a $\mu$-nullset $N_{\mu}$ and a $\nu$-nullset $N_{\nu}$ such that
  $$
     B \subseteq (N_{\mu}\times \R) \cup (\R \times N_{\nu}) \cup \left(\Delta \cup \bigcup_{k\geq1} I_{k}\times J_{k}\right)^{c},
  $$
  where $\Delta=\{(x,x)\in\R^{2}:\, x\in\R\}$ is the diagonal.
\end{theorem}

The main step in the proof is the following construction.

\begin{lemma}\label{le:dominatingMartingaleLaw}
  Let $\mu\leq_{c}\nu$ be irreducible and let $\pi$ be a finite measure on~$\R^{2}$ whose marginals $\pi_{1},\pi_{2}$ satisfy\footnote{By $\pi_{1}\leq \mu$ we mean that $\pi_{1}(A)\leq \mu(A)$ for every Borel set $A\subseteq \R$.} $\pi_{1}\leq \mu$ and $\pi_{2}\leq\nu$. Then, there exists $P\in\cM(\mu,\nu)$ such that $P$ dominates $\pi$ in the sense of absolute continuity.
\end{lemma}

\begin{proof} 
  Let $(I,J)$ be the domain of $(\mu,\nu)$.  We may assume that $\mu,\nu$ are probability measures; in particular, $I\neq\emptyset$.
  
  (i) We first show the result under the additional hypothesis that $\pi$ is supported on an compact rectangle $K\times L\subseteq I\times J$.
  
  Writing $I=(l,r)$, the definition of $(I,J)$ implies that $\nu$ assigns positive mass to any neighborhood of $l$. Since $K$ is compact, it has positive distance to $l$ and we can find a compact set $B_{-}\subseteq J$ with $\nu(B_{-})>0$ to the left of $K$; i.e., $l<y<x$ for all $y\in B_{-}$ and $x\in K$. Similarly, we can find a compact $B_{+}\subseteq J$ with positive mass to the right of $K$. 
  Let 
  $$
    \pi=\pi_{1} \otimes \kappa 
  $$
  be a disintegration of $\pi$; we may choose a version of the kernel $\kappa(x,dy)$ that is concentrated on $L$ for all $x\in K$. We shall now change the mean of $\kappa(x)$ such as to render it a martingale kernel. Indeed,  let us introduce a kernel $\kappa'$ of the form
  $$
    \kappa'(x,dy) = \frac{\kappa(x,dy) + s_{-}(x) \nu(dy)|_{B_{-}}+ s_{+}(x) \nu(dy)|_{B_{+}} }{ c(x)},\quad x\in K.
  $$
  Here $c(x)\geq1$ is the normalizing constant such that $\kappa'(x,dy)$ is a stochastic kernel. Moreover, for $x$ such that the mean of $\kappa (x)$ is smaller or equal to $x$, we set $s_{-}(x):=0$ and define $s_{+}(x)$ as the unique nonnegative scalar such that the mean of $\kappa'(x)$ equals $x$, and analogously in the opposite case. Note that $s_{\pm}$ are well-defined because $B_{\pm}$ is at a positive distance to the left (resp.\ right) of $x\in K$. Then,
  $$
    \pi':= \frac{\nu(B_{-})\wedge \nu(B_{+})}{3} \pi_{1}\otimes \kappa'
  $$
  is a martingale measure with $\pi'\gg \pi$ and its marginals  $\mu',\nu'$ satisfy $\mu'\leq \pi_{1}\leq\mu$ as well as $\nu'\leq \nu$; the latter is due to $\pi_{1}(\R)\leq\mu(\R)=1$ and
  $$
  \kappa'(x) \leq \nu(B_{-})^{-1}\, \nu|_{B_{-}} + \nu(B_{+})^{-1}\, \nu|_{B_{+}} + \kappa(x)
  $$
  and
  $$
   \pi_{1}\otimes \nu|_{B_{-}}\; + \;\pi_{1}\otimes  \nu|_{B_{+}}\; + \;\pi_{1}\otimes\kappa\;\leq \; 3\nu.
  $$
  We also note that 
  \begin{equation}\label{eq:piPrimeConcentrated}
  \mbox{$\pi'$  is concentrated on a compact square $K\times L'$}
  \end{equation}
  where $L'\subseteq J$ is the convex set generated by $B_{-}$ and $B_{+}$. 
  It remains to find $P\in\cM(\mu,\nu)$ such that $P\gg \pi'$. 
  
  (a) We first consider the case where $I=J$. Since $u_{\nu}-u_{\mu}$ is continuous and strictly positive on $I$, this difference is uniformly bounded away from zero on the compact set $L'\subseteq I$. On the other hand, the continuous function $u_{\nu'}-u_{\mu'}$ is uniformly bounded on $L'$. Hence, there is $0<\eps <1$ such that 
  $$
    u_{\mu} - \eps u_{\mu'} \leq u_{\nu} - \eps u_{\nu'}\quad \mbox{on} \quad L',
  $$
  but then this inequality extends to the whole of $\R$ because $u_{\mu'}=u_{\nu'}$ outside of $L'$, due to~\eqref{eq:piPrimeConcentrated}. Noting also that $u_{\mu} - \eps u_{\mu'}=u_{\mu- \eps \mu'}$, we thus have
  $$
    \mu - \eps \mu' \leq_{c} \nu - \eps \nu',
  $$
  and these are nonnegative measures due to $\mu'\leq \mu$ and  $\nu'\leq \nu$. Hence, $\cM(\mu - \eps \mu',\nu - \eps \nu')$ is nonempty; cf.\ Proposition~\ref{pr:convexOrder}. Let~$\pi_{\eps}$ be any element of that set and define
  $$
    P = \eps\mu'(\R)^{-1} \pi' + \pi_{\eps}.
  $$
  By construction, $P$ is an element of $\cM(\mu,\nu)$ and $P\gg\pi'\gg\pi$.
    
  (b) Next, we discuss the case where $\nu$ has an atom at one or both of the endpoints of $I$. Suppose that $\nu(\{r\})>0$; then $L'$ may touch the right boundary of $J$ and we need to give a different argument for the existence of $\eps>0$ as above, since $u_{\nu}-u_{\mu}$ need no longer be bounded away from zero on $L'$. However, the left derivatives satisfy $d^{-}u_{\nu}(r)<d^{-} u_{\mu}(r)$ by Remark~\ref{rk:irredAtoms}, and similarly at $l$ if $\nu(\{l\})>0$. Recalling that the derivatives of any potential function---and in particular of $u_{\mu'}$ and $u_{\nu'}$---are uniformly bounded by the total mass of the corresponding measure, we see that we can still find $\eps>0$ such that $u_{\mu} - \eps u_{\mu'} \leq u_{\nu} - \eps u_{\nu'}$. The rest is as above.
  
  (ii) Finally, we treat the general case. As $\pi_{1}\leq \mu$ and $\pi_{2}\leq\nu$, the measure~$\pi$ is necessarily concentrated on $I\times J$. We can cover $I\times J$ with a sequence $(Q_{n})_{n\geq1}$ of compact rectangles $Q_{n}\subseteq I\times J$ and define measures~$\pi^{n}$ supported by $Q_{n}$ such that $\pi=\sum \pi^{n}$. For each $n$, our construction in (i) yields a martingale transport plan $\pi^{n} \ll P^{n}\in \cM(\mu,\nu)$, and then $P=\sum 2^{-n}P^{n}$ satisfies the requirement of the lemma.
\end{proof}

\begin{corollary}\label{co:polarSetsIrred}
  The pair $\mu\leq_{c}\nu$ is irreducible if and only if the $\Pi(\mu,\nu)$-polar sets and the $\cM(\mu,\nu)$-polar sets coincide.
\end{corollary}

\begin{proof}
  If $\mu\leq_{c}\nu$ is irreducible, the conclusion is an immediate consequence of the preceding lemma.  
  Conversely, suppose that $\mu\leq_{c}\nu$ is not irreducible; that is, there exists $x\in\R$ such that $u_{\mu}(x)=u_{\nu}(x)$ and $\mu(-\infty,x)>0$ and $\mu(x,\infty)>0$. Then, the set $(-\infty,x)\times (x,\infty)$ is $\cM(\mu,\nu)$-polar (Proposition~\ref{pr:decomp}) but not $\Pi(\mu,\nu)$-polar (Proposition~\ref{pr:polarDescrClassical}).
\end{proof}

\begin{proof}[Proof of Theorem~\ref{th:polarDescr}]
  By Proposition~\ref{pr:decomp} and Corollary~\ref{co:polarSetsIrred}, a Borel set $B$ is $\cM(\mu,\nu)$-polar if and only if $B\cap (I_{k}\times J_{k})$ is $\Pi(\mu_{k},\nu_{k})$-polar for all $k\geq1$ and $B\cap \Delta$ is $P_{0}$-null. The result now follows by applying Proposition~\ref{pr:polarDescrClassical} for each $k\geq1$.
\end{proof}

\section{A Generalized Integral}\label{se:generalizedIntegral}

\subsection{Integral of a Concave Function}

Let $\mu\leq_{c}\nu$ be irreducible with domain $(I,J)$ and let $\chi: J\to \R$ be a concave function\footnote{In fact, we will not need irreducibility for the results of this section, except for  Example~\ref{ex:finiteIntegralSuffCondAtoms} and Remark~\ref{rk:atomsEstimate}. Moreover, we could allow $\chi$ to take the value $-\infty$ on $J\setminus I$.}. We assume that $I\neq\emptyset$. Our first aim is to define the difference $\mu(\chi)-\nu(\chi)$. Indeed, $\mu(\chi)$ and $\nu(\chi)$ are well defined in $[-\infty,\infty)$ as $\chi^{+}$ has linear growth, but we need to elaborate on the difference. There are (at least) three natural definitions, and we shall see that they all yield the same value. To that end, note that $\chi$ is continuous on $I$ by concavity, but may have downward jumps at the boundary $J\setminus I$. We denote the absolute magnitude of the jump at $y$ by $|\Delta\chi(y)|$.

\begin{enumerate}
	\item[(1)] 
	\emph{Approximation.} Let $I_{n}$ be a sequence of open, bounded intervals increasing strictly to $I$ (i.e., $I\setminus I_{n}$ has two components for all $n$) and consider the concave, linearly growing functions $\chi_{n}: J\to \R$ defined by the following conditions: $\chi_{n}=\chi$ on $I_{n}$ and on $J\setminus I$, whereas $\chi_{n}$ is affine on each component of $I\setminus I_{n}$, with continuous first derivative at the endpoints of $I_{n}$. Then, $\mu(\chi_{n})$ and $\nu(\chi_{n})$ are both finite and we set
	\begin{equation}\label{eq:defIntegralByApprox}
	  \cI_{1}(\chi,\mu-\nu) := \lim_{n\to\infty} [\mu(\chi_{n}) - \nu(\chi_{n})].
	\end{equation}
	We shall see below that the limit exists in $[0,\infty]$.
	
	\item[(2)]
	\emph{Integration by Parts.} Let $-\chi''$ be the (locally finite) second derivative measure of the convex function $-\chi$ on $I$ and set
	\begin{equation}\label{eq:defIntegralByIP}
	  \cI_{2}(\chi,\mu-\nu) := \frac12 \int_{I} (u_{\mu} - u_{\nu}) \,d\chi'' \, +   \int_{J\setminus I} |\Delta\chi| \,d\nu.
	\end{equation}
  As $u_{\mu} \leq u_{\nu}$ and $\chi<\infty$, this quantity is well defined in $[0,\infty]$.
  
  \item[(3)]
  \emph{Disintegration.} Fix an arbitrary $P\in\cM(\mu,\nu)$ and consider a disintegration $P=\mu\otimes \kappa$; then we have $\int \chi(y)\, \kappa(x,dy)\leq \chi(x)$ for $\mu$-a.e.\ $x\in I$ by Jensen's inequality. Thus,
	$$
	  \cI_{3}(\chi,\mu-\nu) := \int_{I} \left[\chi(x) - \int_{J} \chi(y)\,\kappa(x,dy)\right]	\mu(dx)
	$$
	is well defined in $[0,\infty]$, and we shall see below that this value is independent of the choice of $P\in\cM(\mu,\nu)$. This definition was already used in~\cite{BeiglbockJuillet.12}.
\end{enumerate}

For future reference, let us recall the following fact about the second derivative measure $\chi''$: after normalizing $\chi$ and its left derivative $\chi'$ such that  $\chi(a)=\chi'(a)=0$ for some $a\in I$ (by adding a suitable affine function),
$$
  \chi(y) = 
     \int_{(l,a)} (y-t)^{-} \, \chi''(dt) + \int_{[a,r)} (y-t)^{+} \, \chi''(dt),\quad y\in I,
$$
where $l,r\in [-\infty,\infty]$ are such that $I=(l,r)$. If $\chi$ is continuous at the boundary of $I$, this identity extends to $y\in J$ by monotone convergence.

\begin{lemma}\label{le:intConcaveWelldef}
  The values $\cI_{i}(\chi,\mu-\nu)$ are well defined in $[0,\infty]$, depend only on $\chi$ and $\mu-\nu$, and coincide for $i=1,2,3$. 
\end{lemma}

\begin{proof}
  By concavity, $\chi$ is continuous on $I$ with possible downward jumps at the boundary. Setting $\bar\chi:=\chi$ on $I$ and extending $\bar\chi$ to $J$ by continuity, we have $\chi = \bar\chi - |\Delta \chi|\1_{J\setminus I}$ where $\bar \chi$ is concave and continuous. By linearity of the $\nu$-integral, it suffices to show the claim for $\bar\chi$; in other words, we may assume that $\chi$ is continuous.

  Suppose first that $\chi\in L^{1}(\mu)\cap L^{1}(\nu)$. Then, it is clear that $\cI_{i}(\chi,\mu-\nu)$ is well defined for $i=1,2,3$ and that $\cI_{1}(\chi,\mu-\nu)=\cI_{3}(\chi,\mu-\nu)$. To see the equality with $\cI_{2}(\chi,\mu-\nu)$, let $a\in I$ be arbitrary. Writing again $I=(l,r)$, we have
  \begin{align*}
    \int_{J} \chi(s) \, (\mu-\nu)(ds) 
  	& =\int_{[l,a)}\int_{(l,a)}(t-s)^+ \,\chi''(dt)\,(\mu-\nu)(ds) \\
	&\phantom{=}\;+ \int_{[a,r]}\int_{[a,r)}(s-t)^+ \,\chi''(dt)\,(\mu-\nu)(ds). %
	\end{align*}	
Applying Fubini's theorem to both integrals and noting that the integrands vanish on certain sets, this can be rewritten as
  $$
    \int_{(l,a)}\int_J(t-s)^+ \,(\mu-\nu)(ds)\,\chi''(dt) \\
	+ \int_{[a,r)}\int_J(s-t)^+ \,(\mu-\nu)(ds)\,\chi''(dt).
	$$
	Substituting $(t-s)^{+} = (s-t)^{+} + t - s$ in the first integral and using that $\mu$ and $\nu$ have the same mass and mean, this equals 
	$$
    \int_{I}\int_J(s-t)^+ \,(\mu-\nu)(ds)\,\chi''(dt).
	$$
	On the other hand, using $|s-t|=2(s-t)^{+} - (s-t)$ yields that
  $$
    (u_{\mu} - u_{\nu})(t) = \int_{J} |s-t|\,(\mu-\nu)(ds) = 2 \int_{J} (s-t)^{+}\,(\mu-\nu)(ds).
  $$
  It follows that $\cI_{2}(\chi,\mu-\nu)=\cI_{3}(\chi,\mu-\nu)$ and that this value depends only on $\chi$ and $\mu-\nu$.
  
  For general $\chi$, define $\chi_{n}\in L^{1}(\mu)\cap L^{1}(\nu)$ as before~\eqref{eq:defIntegralByApprox}; the above establishes that the values of $\cI_{i}(\chi_{n},\mu-\nu)$ coincide for each $n$. Noting that $\chi_{n}$ decreases to $\chi$ stationarily and that $\chi_{n+1}-\chi_{n}$ is concave, monotone convergence entails that $\cI_{i}(\chi_{n},\mu-\nu)\to \cI_{i}(\chi,\mu-\nu)$ for $i=2,3$, and in particular these limits coincide. It now follows that the limit defining $\cI_{1}(\chi,\mu-\nu)$ must exist and have the same value.
\end{proof}

\begin{definition}\label{de:diffInt}
  We write $(\mu-\nu)(\chi)$ for the common value of $\cI_{i}(\chi,\mu-\nu)$, $i=1,2,3$.
\end{definition}

As the notation suggests, we have $(\mu-\nu)(\chi)= \mu(\chi) - \nu(\chi)$ as soon as at least one of the latter integrals is finite---this follows from the representation~\eqref{eq:defIntegralByApprox}. However, it may happen that $(\mu-\nu)(\chi)$ is finite but $\mu(\chi)=\nu(\chi)=-\infty$. The following remark elaborates on this.

\begin{remark}\label{rk:finiteIntegralSuffCond}
	Suppose that $(\mu-\nu)(\chi)$ is finite; then $\mu(\chi)$ and $\nu(\chi)$ are either both infinite or both finite. Thus, one sufficient condition for their finiteness is that the support of $\mu$ be a compact subset of $I$. A more general condition is the existence of a constant $C\geq1$ such that
	\begin{equation}\label{eq:finiteIntegralSuffCond}
	  u_{\nu} - u_{\delta_{m}} \leq C(u_{\nu} - u_{\mu}),
	\end{equation}
	where $m\in I$ is the barycenter of $\mu$. Indeed, by~\eqref{eq:defIntegralByIP}, this implies that
	$$
	  (\delta_{m}-\nu)(\chi) \leq  C (\mu-\nu) (\chi);
	$$
	thus, $\nu(\chi)>-\infty$ if the right-hand side is finite. One can formulate a similar sufficient condition by substituting $\delta_{m}$ with any measure $\bar\mu$ satisfying $\bar\mu\leq_{c}\nu$ and $\bar\mu(\chi)>-\infty$.
\end{remark}

\begin{example}\label{ex:finiteIntegralSuffCondGaussian}
  Let $\mu,\nu$ be Gaussian with the same mean and variances $\sigma^{2}_{\mu}< \sigma^{2}_{\nu}$. Then, a direct computation shows that $\mu\leq_{c}\nu$ is irreducible and Condition~\eqref{eq:finiteIntegralSuffCond} is satisfied.
\end{example}

It turns out that atoms at the endpoints of $I$ are helpful in terms of integrability.

\begin{example}\label{ex:finiteIntegralSuffCondAtoms}
  Suppose that $\mu\leq_{c}\nu$ is irreducible with domain $(I,J)$. If $I$ is bounded and $\nu$ has atoms at both endpoints of $I$, then~\eqref{eq:finiteIntegralSuffCond} is satisfied. Indeed, $u_{\nu} > u_{\mu}$ on $I$ and the slopes are separated at the endpoints (cf.\ Remark~\ref{rk:irredAtoms}) so that $(u_{\nu}-u_{\delta})/(u_{\nu}-u_{\mu})$ has a positive limit at the boundary. Using these two facts, \eqref{eq:finiteIntegralSuffCond} follows.
\end{example}

Very much in the same spirit, we have the following estimate related to the preceding example. 

\begin{remark}\label{rk:atomsEstimate}
  Let $\mu\leq_{c}\nu$ be irreducible with domain $(I,J)$, let $I$ have a finite right endpoint $r$ and let $\chi: J\to\R$ be a concave function such that $\chi(a)=\chi'(a)=0$, where $a\in I$ is the common barycenter of~$\mu$ and~$\nu$.  In particular, $\chi\leq0$ and $\chi\1_{[a,\infty)}$ is concave. If $\nu$ has an atom at $r$, then
  \begin{equation}\label{eq:atomEstimateForChi}
    \chi(r) \geq -\frac{C}{\nu(\{r\})} \int_{[a,\infty)} \chi \,d(\mu-\nu),
  \end{equation}
  with a constant $C\geq0$ depending only on $\mu,\nu$.
  
  Indeed, as in Example~\ref{ex:finiteIntegralSuffCondAtoms}, Remark~\ref{rk:irredAtoms} implies that there exists $C$ such that~\eqref{eq:finiteIntegralSuffCond} holds on $[a,\infty)$. As a consequence,
  $$
    -\chi(r)\nu(\{r\}) \leq -\int_{[a,\infty)} \chi \,d\nu =  \int_{[a,\infty)} \chi \,d(\delta_{a} -\nu) \leq C \int_{[a,\infty)} \chi \,d(\mu-\nu),
  $$
  where we have applied Lemma~\ref{le:intConcaveWelldef} to $\chi\1_{[a,\infty)}$. 
\end{remark}

\subsection{Integrability Modulo Concave Functions}

Our next aim is to define expressions of the form $\mu(\varphi)+\nu(\psi)$ in a situation where the individual integrals are not necessarily finite. We continue to assume that $\mu\leq_{c}\nu$ is irreducible with domain $(I,J)$.

\begin{definition}\label{de:ModConcaveIntegral}
  Let $\varphi: I\to \overline \R$ and $\psi: J\to \overline \R$ be Borel functions. If there exists a concave function $\chi: J\to \R$ such that $\varphi-\chi \in L^{1}(\mu)$ and  $\psi+\chi \in L^{1}(\nu)$, we say that $\chi$ is a \emph{concave moderator} for $(\varphi,\psi)$ and set
  $$
  \mu(\varphi)+\nu(\psi) := \mu(\varphi-\chi)+\nu(\psi+\chi) + (\mu-\nu)(\chi) \,\in (-\infty,\infty],
  $$
  where $(\mu-\nu)(\chi)$ was introduced in Definition~\ref{de:diffInt}.
\end{definition}

\begin{remark}\label{rk:integralDefIndep}
	The preceding definition is independent of the choice of the concave moderator $\chi$. Indeed, suppose there is another concave function $\bar\chi$ such that  $\varphi-\bar\chi \in L^{1}(\mu)$ and  $\psi+\bar\chi \in L^{1}(\nu)$, then it follows that $\chi-\bar\chi\in L^{1}(\mu)\cap L^{1}(\nu)$. Using, for instance, the representation~\eqref{eq:defIntegralByApprox}, we see that
	$$
	  (\mu-\nu)(\chi) - (\mu-\nu)(\chi-\bar\chi) = (\mu-\nu)(\bar\chi)
	$$
	and now it follows that 
	$$
	  \mu(\varphi-\chi)+\nu(\psi+\chi) + (\mu-\nu)(\chi) = \mu(\varphi-\bar\chi)+\nu(\psi+\bar\chi) + (\mu-\nu)(\bar\chi)
	$$
	as desired.
\end{remark}

\begin{definition}\label{de:ModConcaveL}
  We denote by $L^{c}(\mu,\nu)$ the space of all pairs of Borel functions $\varphi: I\to \overline \R$ and $\psi: J\to \overline \R$ which admit a concave moderator $\chi$ such that $(\mu-\nu)(\chi)<\infty$.
\end{definition}

In particular, $\mu(\varphi)+\nu(\psi)$ is well defined and finite for $(\varphi,\psi)\in L^{c}(\mu,\nu)$, and has the usual value if $(\varphi,\psi)\in L^{1}(\mu)\times L^{1}(\nu)\subseteq L^{c}(\mu,\nu)$.

The following sanity check confirms that $\mu(\varphi)+\nu(\psi)$ has the good value in the context of martingale transport.

\begin{remark}\label{rk:integralOfSuperhedge}
  Let $(\varphi,\psi)\in L^{c}(\mu,\nu)$ and let $h: I\to\R$ be Borel. If $\varphi(x) + \psi(y) + h(x)(y-x)$ is bounded from below on $I\times J$, then
  $$
    \mu(\varphi)+\nu(\psi) = P[\varphi(X) + \psi(Y) + h(X)(Y-X)] 
  $$
  for any $P\in\cM(\mu,\nu)$.
\end{remark}

\begin{proof}
  Let $\chi$ be a concave moderator for $(\varphi,\psi)$. We may suppose that $0$ is a lower bound, so that
  $$
    (\varphi- \chi)(X) + (\psi + \chi)(Y) + \chi(X)- \chi(Y)+ h(X)(Y-X) \geq 0.
  $$
  As the first two terms are $P$-integrable, the negative part of the remaining expression is $P$-integrable and
  $P[\varphi(X) + \psi(Y) + h(X)(Y-X)]$ equals
  $$
    \mu(\varphi- \chi) + \nu(\psi + \chi) + P[\chi(X)- \chi(Y)+ h(X)(Y-X)].
  $$  
  Let $P=\mu\otimes \kappa$ be a disintegration of $P$; then by the linear growth of~$\chi^{+}$, the following integrals are well defined and equal,
  $$
    \int [\chi(x)- \chi(y)+ h(x)(y-x)] \,\kappa(x,dy) = \int [\chi(x)- \chi(y)] \,\kappa(x,dy)
  $$
  for $\mu$-a.e.\ $x\in I$. As the negative part of $\chi(X)- \chi(Y)+ h(X)(Y-X)$ is $P$-integrable, Fubini's theorem (for kernels) yields 
  $$
    P[\chi(X)- \chi(Y)+ h(X)(Y-X)] = \iint [\chi(x)- \chi(y)] \,\kappa(x,dy)\, \mu(dx)
  $$
  and the right-hand side equals $(\mu-\nu)(\chi)$ by Lemma~\ref{le:intConcaveWelldef}.
\end{proof}

\section{Closedness on an Irreducible Component}\label{se:closednessOnIrred}

In this section, we analyze the dual problem on a single component; that is, we continue to assume that $\mu\leq_{c}\nu$ is irreducible with domain $(I,J)$. 

\begin{definition}\label{de:dualDomainPointw}
  Let $f: I\times J \to [0,\infty]$. We denote by $\cD^{c,pw}_{\mu,\nu}(f)$ the set of all Borel functions $(\varphi,\psi,h): \R\to\overline\R \times \overline\R \times \R$ such that $(\varphi,\psi)\in L^{c}(\mu,\nu)$ and 
  $$
    \varphi(x) + \psi(y) + h(x)(y-x) \geq f(x,y),\quad (x,y)\in I\times J.
  $$
  Moreover, we denote by $\cD^{1,pw}_{\mu,\nu}(f)$ the subset of all $(\varphi,\psi,h)\in\cD^{c,pw}_{\mu,\nu}(f)$ with $\varphi\in L^{1}(\mu)$ and $\psi\in L^{1}(\nu)$.
\end{definition}

We emphasize that in this definition, the inequality is stated in the pointwise (``pw'') sense. For later reference, we also note that there are two degrees of freedom in the choice of $(\varphi,\psi,h)$. Namely, given constants $c_{1},c_{2}\in\R$, the triplet $(\varphi,\psi,h)$ belongs to $\cD^{c,pw}_{\mu,\nu}(f)$ if and only if the the triplet
\begin{equation}\label{eq:affineNormalization}
 \tilde\varphi(x)=\varphi(x)+c_{1}+c_{2}x,\quad \tilde\psi(y)=\psi(y)-c_{1}-c_{2}y,\quad \tilde h(x)=h(x)+c_{2}
\end{equation}
does, and then $\mu(\varphi)+\nu(\psi) = \mu(\tilde\varphi)+\nu(\tilde\psi)$.

The goal of the present section is the following closedness result for $\cD^{c,pw}_{\mu,\nu}(f)$; it is at the very heart of our duality and existence theory.

\begin{proposition}\label{pr:closednessIrred}
   Suppose that  $\mu\leq_{c}\nu$ is irreducible with domain $(I,J)$, let~$f,f_{n}: I\times J \to [0,\infty]$ be such that $f_{n}\to f$ pointwise and let $(\varphi_{n},\psi_{n},h_{n})\in \cD^{c,pw}_{\mu,\nu}(f_{n})$ satisfy $\sup_{n} \{\mu(\varphi_{n})+\nu(\psi_{n})\}<\infty$. Then, there exist
   $$
     (\varphi,\psi,h)\in \cD^{c,pw}_{\mu,\nu}(f) \quad\mbox{such that}\quad \mu(\varphi)+\nu(\psi)\leq \liminf_{n\to\infty} \{\mu(\varphi_{n})+\nu(\psi_{n})\}.
   $$
\end{proposition}

The irreducible pair $\mu\leq_{c}\nu$ is fixed for the rest of this section, so let us simplify the notation to 
$$
  \cD^{c}(f):=\cD^{c,pw}_{\mu,\nu}(f).
$$
As a first step towards the proof of Proposition~\ref{pr:closednessIrred}, we introduce concave functions which will control simultaneously $\varphi_{n}$ and $\psi_{n}$, in the sense of one-sided bounds.

\begin{lemma}\label{le:passageToChi}
  Let $(\varphi, \psi,h)\in \cD^{c}(0)$. Then, there exists a concave moderator $\chi: J\to \R$ for $(\varphi, \psi)$ such that 
  $$
    \chi\leq\varphi \mbox{ on } I,\quad -\chi \leq \psi \mbox{ on } J.
  $$
  In particular, $\mu(\chi)+\nu(-\chi)\leq \mu(\varphi)+\nu(\psi)$.
\end{lemma}

\begin{proof}
  The function 
  $$
    \chi(y) := \inf_{x\in I} \,[\varphi(x) + h(x)(y-x)],\quad y\in J
  $$
  is concave as an infimum of affine functions, and $(\varphi,\psi)\in L^{c}(\mu,\nu)$ implies that $\varphi<\infty$ on a nonempty set, so that $\chi<\infty$ everywhere on $J$. Moreover, we clearly have $\chi\leq \varphi$ on $I$. Our assumption that 
  \begin{equation}\label{eq:zeroSuperhedge}
    \varphi(x)+\psi(y)+h(x)(y-x) \geq 0,\quad (x,y)\in I\times J
  \end{equation}
  shows that $\chi \geq -\psi$ on $J$. Since $(\varphi,\psi)\in L^{c}(\mu,\nu)$, the set $\{\psi < \infty\}$ is dense in $\supp(\nu)$, 
  and by concavity it follows that $\chi>-\infty$ on the interior of the convex hull of $\supp(\nu)$; that is, on the interval $I$. Moreover, $\{\psi<\infty\}$ must contain any atom of $\nu$ and in particular $J\setminus I$, so that $\chi>-\infty$ on $J$.
  
  Setting $\bar\varphi:=\varphi - \chi \geq 0$ and $\bar\psi:=\psi+\chi \geq 0$, we can write~\eqref{eq:zeroSuperhedge} as
  $$
    \bar\varphi(x) + \bar\psi(y) + [\chi(x)-\chi(y)] + h(x)(y-x) \geq0, \quad (x,y)\in I\times J.
  $$
  Let  $P = \mu \otimes \kappa$ be a disintegration of some $P\in\cM(\mu,\nu)$. For fixed $x\in I$, all four terms above are bounded from below by linearly growing functions. It follows that for $\mu$-a.e.\ $x\in I$, the integral of the left-hand side with respect to $\kappa(x,dy)$ can be computed term-by-term, which yields 
  $$
    \bar\varphi(x) + \int \bar\psi(y) \,\kappa(x,dy) + \int [\chi(x)- \chi(y)] \,\kappa(x,dy).
  $$
  These three terms are nonnegative, and thus the integral with respect to $\mu$ can again be computed term-by-term. By Fubini's theorem and Lemma~\ref{le:intConcaveWelldef}, it follows that
  \begin{equation}\label{eq:canonicalModeratorIntegrable}
    P[\bar\varphi(X) + \bar\psi(Y) + [\chi(X)-\chi(Y)] + h(X)(Y-X)] = \mu(\bar\varphi) + \nu(\bar\psi)+ (\mu - \nu)(\chi). 
  \end{equation}
  Of course, the left-hand side is also equal to $P[\varphi(X) + \psi(Y) + h(X)(Y-X)]$ and therefore finite by Remark~\ref{rk:integralOfSuperhedge}. Thus, the right-hand side is finite as well. As a result, $(\bar\varphi,\bar\psi)\in L^{c}(\bar\mu,\bar\nu)$ with concave moderator $\chi$, and 
  $$
    \mu(\varphi)+\nu(\psi) = \mu(\bar\varphi) + \nu(\bar\psi)+ (\mu - \nu)(\chi) \geq (\mu - \nu)(\chi) = \mu(\chi)+\nu(-\chi)
  $$
  as desired.
\end{proof}

Let us record a variant of the preceding construction for later use. 

\begin{remark}\label{rk:reverseIntegrability}
  Let $(\varphi,\psi,h): \R\to  (-\infty,\infty] \times (-\infty,\infty]  \times \R$ be Borel functions such that 
  $$
    \varphi(x) + \psi(y) + h(x)(y-x) \geq 0,\quad (x,y)\in I\times J.
  $$  
  Then, $(\varphi,\psi)\in L^{c}(\mu,\nu)$ if and only if $P[\varphi(X) + \psi(Y) + h(X)(Y-X)]<\infty$ for some (and then all) $P\in\cM(\mu,\nu)$.
\end{remark}

\begin{proof}
  The ``only if'' statement is immediate from Remark~\ref{rk:integralOfSuperhedge}. For the converse, let $P[\varphi(X) + \psi(Y) + h(X)(Y-X)]<\infty$ for some $P\in\cM(\mu,\nu)$; then $\varphi$ is finite $\mu$-a.s.\ and $\psi$ is finite $\nu$-a.s. We can then follow the proof of Lemma~\ref{le:passageToChi} up to~\eqref{eq:canonicalModeratorIntegrable} to define a concave function $\chi: J\to \R$ such that $\bar\varphi:=\varphi - \chi \geq 0$ and $\bar\psi:=\psi+\chi \geq 0$ and
  $$
    P[ \varphi(X) +  \psi(Y)  + h(X)(Y-X)] = \mu(\bar\varphi) + \nu(\bar\psi)+ (\mu - \nu)(\chi). 
  $$
  Since the left-hand side is finite, the three (nonnegative) terms on the right-hand side are finite as well; that is, $(\varphi,\psi)\in L^{c}(\mu,\nu)$ with concave moderator~$\chi$.
\end{proof}

Our second tool for the main result is a compactness principle for concave functions. Irreducibility is crucial for its proof, so let us restate this standing condition. The notation $\chi_{n}'$ refers to the left derivative (say).

\begin{proposition}\label{pr:concaveCompactness}
  Let $\mu\leq_{c}\nu$ be irreducible with domain $(I,J)$ and let $a\in I$ be the common barycenter of $\mu$ and $\nu$. Let $\chi_{n}: J\to \R$ be concave functions such that 
  $$
    \chi_{n}(a)=\chi_{n}'(a)=0\quad \mbox{and}\quad \sup_{n\geq 1} (\mu-\nu)(\chi_{n})<\infty.
  $$
  There exists a subsequence $\chi_{n_{k}}$ which converges pointwise on $J$ to a concave function $\chi: J\to \R$, and $(\mu-\nu)(\chi) \leq \liminf_{k} (\mu-\nu)(\chi_{n_{k}})$.
\end{proposition}

\begin{proof}
  By our assumption, $(\mu-\nu)(\chi_{n})$ is bounded uniformly in $n$. In view of~\eqref{eq:defIntegralByIP}, this implies that there exists a constant $C>0$ such that 
  $$
    0\leq \int_{I} (u_{\mu}-u_{\nu})\, d\chi''_{n} \leq C \quad \mbox{and} \quad  0\leq |\Delta\chi_{n}|\leq C,
  $$
  where we have used that $J\setminus I$ consists of (at most two) atoms of $\nu$ and $|\Delta\chi_{n}|=0$ on $I$. By the same fact, we thus have
  \begin{equation}\label{eq:zeroDeltaLimit}
    \lim_{k} |\Delta\chi_{n_{k}}|  = \liminf_{n} |\Delta\chi_{n}|
  \end{equation}
  for a suitable subsequence $\chi_{n_{k}}$; we may assume that $n_{k}=k$.
  Moreover, the first inequality shows that the sequence of finite measures defined by $(u_{\mu}-u_{\nu})\, d\chi''_{n}$  is bounded and thus relatively compact for the weak topology induced by the compactly supported continuous functions on $I$. Recalling that $u_{\nu}-u_{\mu}$ is continuous and strictly positive on $I$, it follows that $(-\chi''_{n})$ is  relatively weakly compact as well. In view of $\chi_{n}'(a)=0$, this implies a uniform bound for the Lipschitz constant of $\chi_{n}$ on any given compact subset of $I$. Using also $\chi_{n}(a)=0$, the Arzela--Ascoli theorem then yields a function $\chi: I\to\R$ such that $\chi_{n}\to\chi$ locally uniformly, after passing to a subsequence. Clearly $\chi$ is concave, and integration by parts shows that $-\chi''_{n}$ converges weakly to the second derivative measure $-\chi''$ associated with $\chi$.
  Approximating $u_{\mu}-u_{\nu}$ from above with compactly supported continuous functions on $I$, we then see that
  $$
    (\mu-\nu)(\chi) = \frac12 \!\int_{I} (u_{\mu}-u_{\nu})\, d\chi'' \leq \liminf_{n\to\infty} \frac12 \!\int_{I} (u_{\mu}-u_{\nu})\, d\chi''_{n} = \liminf_{n\to\infty} (\mu-\nu)(\chi_{n}).
  $$
  Together with~\eqref{eq:zeroDeltaLimit}, we can define $\chi$ on $J$ and the result follows via~\eqref{eq:defIntegralByIP}.
\end{proof}

We can now derive the main result of this section.

\begin{proof}[Proof of Proposition~\ref{pr:closednessIrred}]
 Since $(\varphi_{n},\psi_{n},h_{n})\in \cD^{c}(f_{n})$ and $f_{n}\geq0$, we can introduce the associated concave functions $\chi_{n}$ as in Lemma~\ref{le:passageToChi}. Normalizing $(\varphi_{n},\psi_{n},h_{n})$ as in~\eqref{eq:affineNormalization}  with suitable constants, we may assume that $\chi_{n}(a)=\chi_{n}'(a)=0$; note that the relations $\chi_{n}\leq \varphi_{n}$ and $-\chi_{n} \leq \psi_{n}$ are preserved. After passing to a subsequence, Proposition~\ref{pr:concaveCompactness} then yields a pointwise limit $\chi: J\to\R$ for the $\chi_{n}$.
  
  Since $\varphi_{n}\geq \chi_{n}\to \chi$, Komlos' lemma (in the form of \cite[Lemma~A1.1]{DelbaenSchachermayer.94} and its subsequent remark) shows that there are $\bar\varphi_{n}\in\conv\{\varphi_{n}, \varphi_{n+1},\dots\}$ which converge $\mu$-a.s., and similarly for $\psi_{n}$. Without loss of generality, we may assume that $\bar\varphi_{n}=\varphi_{n}$, and similarly for~$\psi_{n}$. Thus, setting
  $$
    \varphi:=\limsup \varphi_{n} \quad\mbox{on}\quad I, \quad\quad  \psi:=\limsup \psi_{n} \quad\mbox{on}\quad J
  $$
  yields Borel functions $\varphi$, $\psi$ such that 
  $$
    \varphi_{n}\to\varphi\quad  \mu\as, \quad \varphi-\chi \geq0\quad\mbox{and} \quad  \psi_{n}\to\psi\quad  \nu\as, \quad \psi+\chi \geq0.
  $$
  Fatou's lemma and Proposition~\ref{pr:concaveCompactness} then show that 
  \begin{align*}
    \mu(\varphi-&\chi) + \nu(\psi+\chi) + (\mu-\nu)(\chi) \\
    & \leq \liminf \mu(\varphi_{n}-\chi_{n}) +  \liminf \nu(\psi_{n}+\chi_{n}) + \liminf  (\mu-\nu)(\chi_{n})\\
    & \leq \liminf [ \mu(\varphi_{n}-\chi_{n}) + \nu(\psi_{n}+\chi_{n}) + (\mu-\nu)(\chi_{n}) ] \\
    & = \liminf [ \mu(\varphi_{n}) + \nu(\psi_{n}) ] < \infty.
  \end{align*}
  In particular, this shows that $(\varphi,\psi)\in L^{c}(\mu,\nu)$ with concave moderator $\chi$, and then the above may be stated  more concisely as
  $$
    \mu(\varphi)+\nu(\psi)\leq \liminf \mu(\varphi_{n})+\nu(\psi_{n}).
  $$
  
  It remains to find $h$.  For any function $g: J\to \overline{\R}$, let $g^{conc}: J\to \overline{\R}$ denote the concave envelope. Given a sequence of such functions $g_{n}$, we have
  $$
    \liminf (g_{n}^{conc}) \geq  (\liminf g_{n})^{conc}
  $$
  as $g_{n}^{conc}\geq g_{n}$ and $\liminf g_{n}^{conc}$ is concave. Moreover, $(\varphi_{n},\psi_{n},h_{n})\in \cD^{c}(f_{n})$ means that
  $
    \varphi_{n}(x)+ h_{n}(x)(y-x)\geq f_{n}(x,y) - \psi_{n}(y)
  $ 
  which implies that 
  $$
    \varphi_{n}(x)+ h_{n}(x)(y-x)\geq [f_{n}(x,\cdot)-\psi_{n}]^{conc}(y) ,\quad (x,y)\in I\times J.
  $$
  Fix $x\in I$; then these two facts yield
  \begin{align*}
    \liminf [\varphi_{n}(x)+ h_{n}(x)(y-x)]
    & \geq \liminf [f_{n}(x,\cdot)-\psi_{n}]^{conc}(y) \\
    & \geq [ \liminf  (f_{n}(x,\cdot)-\psi_{n})]^{conc}(y) \\
    & \geq [ f(x,\cdot)-\psi]^{conc}(y) \\
    &=: \hat\varphi(x,y)
  \end{align*}
   for all $y\in J$, and for the specific choice $y=x$ we obtain that
  $$
    \varphi(x) \geq \liminf \varphi_{n}(x) \geq \hat\varphi(x,x).
  $$
  As $\nu\{\psi=\infty\}=0$ and $f>-\infty$, we have $\hat\varphi(x,y)>-\infty$ for all $y\in J$. If $x\notin N:=\{\varphi=\infty\}$, the above inequalities also show that $\hat\varphi(x,x)<\infty$ and as a result, the concave function $\hat\varphi(x,\cdot)$ is finite on $J$ and admits a left derivative 
  $$
   h(y):=d^{-}\hat{\varphi}(x,\cdot)(y)\in\R, \quad y\in I.
  $$
  By concavity, it follows that
  $$
    \varphi(x) + h(x)(y-x) \geq \hat \varphi(x,x)  + h(x)(y-x) \geq \hat \varphi(x,y) \geq f(x,y) - \psi(y)
  $$
  for all $y\in J$. Setting $h:=0$ on $N$, we then have $\varphi(x) + \psi(y) + h(x)(y-x) \geq f(x,y)$ for all $(x,y)\in I\times J$, because the left-hand side is infinite for $x\in N$. Thus, $(\varphi,\psi,h)\in \cD^{c}(f)$ and the proof is complete.
\end{proof}

\section{Duality on an Irreducible Component}\label{se:dualityOnIrred}

Let $\mu\leq_{c}\nu$ be irreducible with domain $(I,J)$.  We define the primal and dual values as follows.

\begin{definition}\label{de:primalAndDualOnIrred}
  Let $f: \R^{2}\to [0,\infty]$. The \emph{primal problem} is 
  $$
    \bS_{\mu,\nu}(f) :=\sup_{P\in\cM(\mu,\nu)} P(f) \;\in [0,\infty],
  $$
  where $P(f)$ refers to the outer integral if $f$ is not measurable. 
  The \emph{dual problem} is
  $$
  \bI^{pw}_{\mu,\nu}(f) :=\inf_{(\varphi,\psi,h)\in \cD^{c,pw}_{\mu,\nu}(f)}  \{\mu(\varphi) + \nu(\psi)\} \;\in [0,\infty].
  $$ 
\end{definition}

The goal of this section is the following duality result; it corresponds to our main result in the case of irreducible marginals. We recall that a function $f: \R^{2}\to [0,\infty]$ is called \emph{upper semianalytic} if the sets $\{f\geq c\}$ are analytic for all $c\in\R$, where a subset of $\R^{2}$ is called analytic if it is the (forward) image of a Borel subset of a Polish space under a Borel mapping. Any Borel function is upper semianalytic and any upper semianalytic function is universally measurable; see, e.g., \cite[Chapter~7]{BertsekasShreve.78} for background.

\begin{theorem}\label{th:dualityIrred}
  Let $\mu\leq_{c}\nu$ be irreducible and let $f: \R^{2}\to [0,\infty]$.
  \begin{enumerate}
  \item If $f$ is upper semianalytic, then $\bS_{\mu,\nu}(f)=\bI^{pw}_{\mu,\nu}(f) \in [0,\infty]$.
  \item If\, $\bI^{pw}_{\mu,\nu}(f)<\infty$, there exists a dual optimizer $(\varphi,\psi,h)\in \cD^{c,pw}_{\mu,\nu}(f)$.
  \end{enumerate}
\end{theorem} 

The proof of Theorem~\ref{th:dualityIrred} is based on Proposition~\ref{pr:closednessIrred}, Choquet's theorem and a separation argument, so let us introduce the relevant terminology. Let $[0,\infty]^{\R^2}$ be the set of all functions $f: \R^{2}\to [0,\infty]$, let $\USA_{+}$ be the sublattice of upper semianalytic functions and let $\cU$ be the sublattice of bounded upper semicontinuous functions; note that $\cU$ is stable with respect to countable infima. A mapping $\bC: [0,\infty]^{\R^2} \to [0,\infty]$ is called a $\cU$-capacity if it is monotone, sequentially continuous upwards on $[0,\infty]^{\R^2}$, and sequentially continuous downwards on~$\cU$. 

We write $\bS(f):=\bS_{\mu,\nu}(f)$ and $\bI(f):=\bI^{pw}_{\mu,\nu}(f)$ for the rest of this section; both of these mappings will turn out to be capacities.

\begin{lemma}\label{le:SisCapacity}
  The mapping $\bS: [0,\infty]^{\R^2}\to [0,\infty]$ is a $\cU$-capacity.
\end{lemma}

\begin{proof}
  Since $\cM(\mu,\nu)$ is weakly compact, this follows by the standard arguments presented, e.g., in \cite[Propositions~1.21, 1.26]{Kellerer.84}.
\end{proof}

Next, we show the absence of a duality gap for upper semicontinuous functions. This result is already known from \cite[Corollary~1.1]{BeiglbockHenryLaborderePenkner.11} which uses a minimax argument and Kellerer's duality theorem~\cite{Kellerer.84} for classical transport. We shall give a direct and self-contained proof based on Proposition~\ref{pr:closednessIrred}.

\begin{lemma}\label{le:uscDuality}
  Let $f\in \cU$; then $\bS(f)=\bI(f)$.
\end{lemma}

\begin{proof}
  Let $f: \R^2\to [0,\infty]$ be bounded and upper semicontinuous; then the inequality 
  \begin{equation}\label{eq:easyIneq}
    \bS(f)\leq \bI(f)
  \end{equation}
  follows from Remark~\ref{rk:integralOfSuperhedge}. Below, we show the converse inequality.
  
  (i) We first prove the result for a class of continuous reward functions. This will be a Hahn--Banach argument, which requires us to introduce a suitable space.
	  
	Recall that $\mu$ has a finite first moment. Thus, by the de la Vall\'ee--Poussin theorem, there exists an increasing function $\zeta_{\mu}: \R_{+}\to\R_{+}$ of superlinear growth such that $x\mapsto \zeta_{\mu}(|x|)$ is $\mu$-integrable. The same applies to $\nu$, and we set 
	$$
	  \zeta(x,y)=1+\zeta_{\mu}(|x|)+\zeta_{\nu}(|y|), \quad (x,y)\in\R^{2}.
	$$
	Let $C_{\zeta}=C_{\zeta}(\R^{2})$ be the vector space of all continuous functions $f: \R^{2}\to\R$ such that $f/\zeta$ vanishes at infinity; this includes all continuous functions of linear growth. We equip $C_{\zeta}$ with the norm $|f|_{\zeta}:=|f/\zeta|_{\infty}$, where $|\cdot|_{\infty}$ is the uniform norm.
	
	Let $f\in C_{\zeta}$. Then, setting $\varphi_{0}(x)=\zeta_{\mu}(|x|)$ and $\psi_{0}(y)=\zeta_{\mu}(|y|)$, we have
	$$
	  -c (1+ \varphi_{0} + \psi_{0}) \leq f \leq c (1+ \varphi_{0} + \psi_{0})
	$$
	for some constant $c$, showing in particular that $\bS(f)$ is finite. Thus, we may assume that $\bS(f)=0$ by a translation. Consider the set 
	$$
	  K=\{g\in C_{\zeta}:\, \bI(g)\leq0\}.
	$$
	This is  a convex cone in $C_{\zeta}$, and Proposition~\ref{pr:closednessIrred} implies that $K$ is closed; here we use that a convergent sequence in $C_{\zeta}$ is uniformly bounded from below by a function of the form $-c (1+ \varphi_{0} + \psi_{0})$.
	
	  Assume for contradiction that $\bI(f)>0$; that is, $f\notin K$. Then the Hahn--Banach theorem and the cone property yield a linear functional $\ell\in C_{\zeta}^{*}$ such that $\ell(K)\subseteq \R_{-}$ and $\ell(f)>0$. We will argue below that $\ell$ can be represented by a finite signed measure $\pi$. Note that $\ell(K)\subseteq \R_{-}$ and the fact that $K$ contains all functions of the form $\varphi(X)-\mu(\varphi)$ with $\varphi\in C_{b}(\R)$ imply that $\ell(\varphi(X))=\mu(\varphi)$ for all $\varphi\in C_{b}(\R)$; i.e., $\mu$ is the first marginal of~$\pi$, and similarly $\nu$ is the second marginal. 
Thus, $\pi\in \Pi(\mu,\nu)$. Moreover, if $h\in C_{b}(\R)$, then the function $h(X)(Y-X)$ is in $C_{\zeta}$ due to its linear growth, and a scaling argument shows that $\ell(h(X)(Y-X))=0$. This implies that $\pi$ is a martingale transport; i.e., $\pi\in\cM(\mu,\nu)$. But now $\pi(f)=\ell(f)>0$ contradicts $\bS(f)=0$, and we have shown that $\bI(f)\leq \bS(f)$. 
	
	It remains to argue that $C_{\zeta}^{*}$ can be represented by finite signed measures. Indeed, $f\mapsto f/\zeta$ is an isomorphism of normed spaces from $C_{\zeta}$ to the usual space $C_{0}(\R^{2})$ of continuous functions vanishing at infinity with the uniform norm. By Riesz' representation theorem, any continuous linear functional on $C_{0}(\R^{2})$ can be represented by a signed measure $m$, and hence any $\ell \in C_{\zeta}^{*}$ can be represented as $\ell(f)=m(f/\zeta)$. Using $1/\zeta \in C_{0}(\R^{2})\subseteq L^{1}(m)$ as a Radon--Nikodym density, $\ell$ is thus represented by the finite signed measure $d\bar m = (1/\zeta)\,dm$. 
	
	(ii) Let $f$ be bounded and upper semicontinuous, then there exist $f_{n}\in C_{b}(\R^{2})\subseteq C_{\zeta}$ decreasing to $f$ and we have $\bS(f_{n})= \bI(f_{n})$ for all $n$ by part~(i) of this proof. As $\bS(f_{n})\to \bS(f)$ by the decreasing continuity of~$\bS$, cf.\ Lemma~\ref{le:SisCapacity}, it remains to show that  $\bI(f_{n})\to \bI(f)$. Since $f \leq f_{n}$, we have $\bI(f) \leq \bI(f_{n})$ for all $n$. On the other hand, \eqref{eq:easyIneq} shows that
	$$
	  \lim \bI(f_{n}) = \lim \bS(f_{n}) = \bS(f) \leq \bI(f)
	$$
	and this completes the proof.
\end{proof}

Our last preparation for the proof of Theorem~\ref{th:dualityIrred} is to show that $\bI$ is  a capacity; again, this is a consequence of the closedness result in Proposition~\ref{pr:closednessIrred}.

\begin{lemma}\label{le:IisCapacity}
  The mapping $\bI: [0,\infty]^{\R^2} \to [0,\infty]$ is a $\cU$-capacity.
\end{lemma}

\begin{proof}
  As $\bI=\bS$ on $\cU$ by Lemma~\ref{le:uscDuality}, Lemma~\ref{le:SisCapacity} already shows that~$\bI$ is sequentially continuous downwards on $\cU$. Let $f,f_{n}\in [0,\infty]^{\R^2}$ be such that $f_{n}$ increases to $f$; we need to show that $\bI(f_{n})\to \bI(f)$. It is clear that $\bI$ is monotone; in particular, $\bI(f)\geq \limsup \bI(f_{n})$, and $\bI(f_{n})\to \bI(f)$ if $\sup_{n} \bI(f_{n})=\infty$.

  Hence, we only need to show $\bI(f) \leq \liminf \bI(f_{n})$ under the condition that $\sup_{n} \bI(f_{n})<\infty$. Indeed, by the definition of $\bI(f_{n})$ there exist $(\varphi_{n},\psi_{n},h_{n})\in \cD^{c,pw}_{\mu,\nu}(f_{n})$ with 
  $$
    \mu(\varphi_{n}) + \nu(\psi_{n}) \leq \bI(f_{n}) +1/n.
  $$
  Proposition~\ref{pr:closednessIrred} then yields $(\varphi,\psi,h)\in \cD^{c,pw}_{\mu,\nu}(f)$ with 
  $$
    \mu(\varphi) + \nu(\psi) \leq \liminf [\bI(f_{n}) +1/n],
  $$
  showing that $\bI(f) \leq \liminf \bI(f_{n})$ as desired.
\end{proof}

We can now deduce the main result of this section.

\begin{proof}[Proof of Theorem~\ref{th:dualityIrred}]
  (i) In view of Lemma~\ref{le:SisCapacity}, %
  Choquet's capacitability theorem shows that
  $$
    \bS(f)=\sup \{\bS(g):\, g\in \cU, \,g\leq f\},\quad f\in\USA_{+}.
  $$
  By Lemma~\ref{le:IisCapacity}, the same approximation formula holds for $\bI$, and as $\bS=\bI$ on $\cU$ by Lemma~\ref{le:uscDuality}, it follows that $\bS=\bI$ on $\USA_{+}$.
  
  (ii) To see that the infimum is attained when it is finite, it suffices to apply Proposition~\ref{pr:closednessIrred} with the constant sequence $f_{n}=f$.
\end{proof}

\section{Main Results}\label{se:mainResults}

\subsection{Duality}

Let $\mu\leq_{c}\nu$ be probability measures in convex order and let $f: \R^{2}\to [0,\infty]$ be a Borel function. We continue to denote the primal problem by
$$
    \bS_{\mu,\nu}(f) :=\sup_{P\in\cM(\mu,\nu)} P(f),
$$
as in the irreducible case. Some more notation needs to be introduced for the dual problem. Let us first recall from Proposition~\ref{pr:decomp} the decompositions
$$
  \mu=\sum_{k\geq0} \mu_{k},\quad \nu=\sum_{k\geq 0} \nu_{k},
$$
where $\mu_{k}\leq_{c}\nu_{k}$ is irreducible with domain $(I_{k},J_{k})$ for $k\geq1$ and $\mu_{0}=\nu_{0}$. Moreover, $P_{0}$ denotes the unique element of $\cM(\mu_{0},\nu_{0})$.

Let $(\varphi,\psi,h): \R\to\overline\R \times \overline\R \times \R$ be Borel. Since $P_{0}$ is concentrated on the diagonal $\Delta$, we have
$$
  \varphi(X)+\psi(Y) + h(X)(Y-X) = \varphi(X)+\psi(X)\quad P_{0}\as;
$$
that is, the function $h$ plays no role and $\varphi,\psi$ enter only through their sum. In fact, the  dual problem associated to $(\mu_{0},\nu_{0})$ is trivially solved, for instance, by setting $\varphi(x)=f(x,x)$ and $\psi=0$. There is no need to use integrability modulo concave functions, but to simplify the notation below, we set 
$$
  L^{c}(\mu_{0},\nu_{0}) := \{(\varphi,\psi): \, \varphi+\psi \in L^{1}(\mu_{0}) \} 
$$
and $\mu_{0}(\varphi)+\nu_{0}(\psi) := \mu_{0}(\varphi+\psi)$ for $(\varphi,\psi)\in L^{c}(\mu_{0},\nu_{0})$. Moreover,  $\cD^{c,pw}_{\mu_{0},\nu_{0}}(f)$ is the set of all $(\varphi,\psi,h)$ with $(\varphi,\psi)\in L^{c}(\mu_{0},\nu_{0})$ and 
$$
  \varphi(x)+\psi(x)\geq f(x,x), \quad x\in I_{0}.
$$
Finally, it will be convenient to define 
$
  \bS_{\mu_{0},\nu_{0}}(f) := P_{0}(f) \equiv \mu_{0}(f(X,X)).
$

We can now introduce the domain for the dual problem on the whole real line.

\begin{definition}\label{de:globalIntegrability}
  Let $L^{c}(\mu,\nu)$ be the set of all Borel functions $\varphi,\psi: \R\to\overline\R$ such that $(\varphi,\psi)\in L^{c}(\mu_{k},\nu_{k})$ for all $k\geq0$ and 
  $$
    \sum_{k\geq0} |\mu_{k}(\varphi)+\nu_{k}(\psi)| <\infty.
  $$
  For $(\varphi,\psi)\in L^{c}(\mu,\nu)$, we define
  $$
    \mu(\varphi)+\nu(\psi) := \sum_{k\geq0} \{\mu_{k}(\varphi)+\nu_{k}(\psi)\} <\infty,
  $$
  and $\cD^{c}_{\mu,\nu}(f)$ is the set of all Borel functions $(\varphi,\psi,h): \R\to\overline\R \times \overline\R \times \R$   
  such that $(\varphi,\psi)\in L^{c}(\mu,\nu)$ and 
  $$
    \varphi(X)+\psi(Y) + h(X)(Y-X)\geq f(X,Y)\quad \cM(\mu,\nu)\qs
  $$
  Finally, 
  $$
  \bI_{\mu,\nu}(f) :=\inf_{(\varphi,\psi,h)\in \cD^{c}_{\mu,\nu}(f)}  \{\mu(\varphi) + \nu(\psi)\} \;\in [0,\infty].
  $$
\end{definition}

We emphasize that the dual domain $\cD^{c}_{\mu,\nu}(f)$ is now defined in the quasi-sure sense. Before making precise the correspondence with the individual components, let us recall that the intervals $J_{k}$ may overlap at their endpoints, so we have to avoid counting certain things twice. Indeed, let $(\varphi_{k},\psi_{k},h_{k})\in \cD^{c,pw}_{\mu_{k},\nu_{k}}(f)$. If $J_{k}$ contains one of its endpoints, it is an atom of~$\nu$ and hence~$\psi_{k}$ is finite on~$J_{k}\setminus I_{k}$. Translating $\psi_{k}$ by an affine function and shifting $\varphi_{k}$ and $h_{k}$ accordingly, cf.\ \eqref{eq:affineNormalization}, we can thus normalize $(\varphi_{k},\psi_{k},h_{k})$ such that
\begin{equation}\label{eq:psiNormalization}
    \psi_{k}=0 \quad \mbox{on}\quad J_{k} \setminus I_{k}.
\end{equation}
On the strength of our analysis of the $\cM(\mu,\nu)$-polar sets, the dual domain can be decomposed as follows.

\begin{lemma}\label{le:reductionToComponentsDual}
  Let $f: \R^{2}\to [0,\infty]$ be Borel, let $\mu\leq_{c}\nu$ and let $\mu_{k},\nu_{k}$ be as in Proposition~\ref{pr:decomp}.
  
  \begin{enumerate}
  \item
  Let $(\varphi_{k},\psi_{k},h_{k})\in \cD^{c,pw}_{\mu_{k},\nu_{k}}(f)$ for $k\geq 1$, normalized as in~\eqref{eq:psiNormalization}, and let $\varphi_{0}(x)=f(x,x)$ and $\psi_{0}=0$. If $\sum_{k\geq0} \{\mu(\varphi_{k})+ \nu(\psi_{k}) \}<\infty$, then
  $$
    \varphi:=\sum_{k\geq0} \varphi_{k} \1_{I_{k}} ,\quad \psi:=\sum_{k\geq1} \psi_{k} \1_{J_{k}}, \quad h:=\sum_{k\geq1} h_{k} \1_{I_{k}}
  $$
  satisfies $(\varphi,\psi,h)\in \cD^{c}_{\mu,\nu}(f)$ and 
  $
    \mu(\varphi)+ \nu(\psi) = \sum_{k\geq0} \mu_{k}(\varphi_{k})+ \nu_{k}(\psi_{k}).
  $
  
  \item 
  Conversely, let $(\varphi,\psi,h)\in \cD^{c}_{\mu,\nu}(f)$. After changing $\varphi$ on a $\mu$-nullset and $\psi$ on a $\nu$-nullset, we have $(\varphi,\psi,h)\in \cD^{c,pw}_{\mu_{k},\nu_{k}}(f)$ for $k\geq 0$ and 
  $$
    \sum_{k\geq0} \{\mu_{k}(\varphi)+ \nu_{k}(\psi)\} = \mu(\varphi)+ \nu(\psi) <\infty.
  $$
  \end{enumerate}
\end{lemma}

\begin{proof}
  In essence, this is a direct consequence of Proposition~\ref{pr:decomp} and Theorem~\ref{th:polarDescr}. For~(i), we note that $\mu(\varphi_{k})+ \nu(\psi_{k})\geq0$ for all $k$, so that the sum is always well defined. Regarding~(ii), let $B$ be the polar set of all $(x,y)$ such that  $\varphi(x)+\psi(y) + h(x)(y-x)< f(x,y)$; note that $B$ is Borel because all these functions are Borel. Then for each $k\geq 1$, the set $B\cap (I_{k} \times J_{k})$ is contained in a union $(N^{k}_{\mu}\times \R) \cup (\R \times N^{k}_{\nu})$, where $N^{k}_{\mu}$ is $\mu$-null and $N^{k}_{\nu}$ is $\nu$-null. We then set $\varphi=\infty$ on $\cup_{k\geq1} N^{k}_{\nu}$ as well as on the $\mu_{0}$-nullset $B\cap \Delta_{0}$. Proceeding analogously with $\psi$, we obtain the desired properties.
\end{proof}

\begin{remark}\label{rk:SameDualIfIrred}
  (i) Suppose that $\mu\leq_{c}\nu$ is irreducible. Then, Lemma~\ref{le:reductionToComponentsDual} implies that the pointwise and the quasi-sure formulation of the dual problem agree:
  $$
    \bI^{pw}_{\mu,\nu}(f)=\bI_{\mu,\nu}(f)
  $$
  if $f=0$ outside the domain $(I,J)$, and otherwise the difference is $P_{0}(f)$ due to our definitions. Without the irreducibility condition, the formulations may differ fundamentally; cf.\ Example~\ref{ex:dualityGap}.

  (ii) As a sanity check on our definitions, we note that
  $$
   \mu(\varphi)+\nu(\psi) = P[\varphi(X) + \psi(Y) + h(X)(Y-X)], \quad P\in\cM(\mu,\nu)
  $$
  whenever $(\varphi,\psi,h)\in \cD^{c}_{\mu,\nu}(f)$ for some $f\geq0$, as a consequence of Lemma~\ref{le:reductionToComponentsDual} and Remark~\ref{rk:integralOfSuperhedge}.
\end{remark}

We can now state our main duality result.

\begin{theorem}\label{th:dualityGlobal}
  Let $f: \R^{2}\to [0,\infty]$ be Borel and let $\mu\leq_{c}\nu$. Then
  $$
    \bS_{\mu,\nu}(f) = \bI_{\mu,\nu}(f) \in [0,\infty].
  $$
  If\, $\bI_{\mu,\nu}(f)<\infty$, there exists an optimizer $(\varphi,\psi,h)\in \cD^{c}_{\mu,\nu}(f)$ for $\bI_{\mu,\nu}(f)$.
\end{theorem}

\begin{proof}
  We first show that $\bS_{\mu,\nu}(f) \leq \bI_{\mu,\nu}(f)$. To this end, we may assume that $\bI_{\mu,\nu}(f)<\infty$, so that there exists some $(\varphi,\psi,h)\in \cD^{c}_{\mu,\nu}(f)$. By Lemma~\ref{le:reductionToComponentsDual}, this induces $(\varphi,\psi,h)\in \cD^{c,pw}_{\mu_{k},\nu_{k}}(f)$, and so the duality result of Theorem~\ref{th:dualityIrred} yields that
  $$
    \bS_{\mu,\nu}(f) \leq \sum_{k\geq0} \bS_{\mu_{k},\nu_{k}}(f) \leq \sum_{k\geq0} \{\mu_{k}(\varphi) +\nu_{k}(\psi)\} =  \mu(\varphi) +\nu(\psi)<\infty.
  $$
  The claim follows as  $(\varphi,\psi,h)\in \cD^{c}_{\mu,\nu}(f)$ was arbitrary.
  
  Next, we prove that $\bS_{\mu,\nu}(f) \geq \bI_{\mu,\nu}(f)$, for which we may assume that $\bS_{\mu,\nu}(f)<\infty$. Then $\bS_{\mu_{k},\nu_{k}}(f)<\infty$ for all $k\geq0$ and by Theorem~\ref{th:dualityIrred} there exist $(\varphi_{k},\psi_{k},h_{k})\in \cD^{c,pw}_{\mu_{k},\nu_{k}}(f)$  such that 
  $$
    \bS_{\mu,\nu}(f) =\sum_{k\geq0} \bS_{\mu_{k},\nu_{k}}(f) = \sum_{k\geq0} \{\mu_{k}(\varphi_{k}) +\nu_{k}(\psi_{k})\}.
  $$
  With the induced $(\varphi,\psi,h)\in \cD^{c}_{\mu,\nu}(f)$ as in Lemma~\ref{le:reductionToComponentsDual}, it follows that 
  $$
    \bS_{\mu,\nu}(f) = \mu(\varphi) +\nu(\psi) \geq \bI_{\mu,\nu}(f) \geq \bS_{\mu,\nu}(f),
  $$
  which shows both the claimed inequality and  that $(\varphi,\psi,h)\in \cD^{c}_{\mu,\nu}(f)$ is optimal for $\bI_{\mu,\nu}(f)$.
\end{proof}

Some remarks on the main result are in order.

\begin{remark}\label{rk:lowerBound}
  The lower bound on $f$ in Theorem~\ref{th:dualityGlobal} can easily be relaxed. Indeed, let $f: \R^{2}\to \overline\R$ be Borel and suppose there exist Borel functions $(\varphi,\psi,h): \R\to\overline\R \times \overline\R \times \R$ such that $\varphi\in L^{1}(\mu)$, $\psi\in L^{1}(\nu)$ and
  $$
    f(X,Y) \geq \varphi(X)+\psi(Y) + h(X)(Y-X) \quad \cM(\mu,\nu)\qs
  $$
  Then, we may apply Theorem~\ref{th:dualityGlobal} to 
  $$
    \bar f := [f(X,Y) - \varphi(X)-\psi(Y) - h(X)(Y-X)]^{+}
  $$
  and the conclusion for $f$ follows, except that now $\bS_{\mu,\nu}(f) = \bI_{\mu,\nu}(f)$ has values in $(-\infty,\infty]$. However, the lower bound cannot be eliminated completely; cf.\ Example~\ref{ex:dualityGapLower}
\end{remark}

We recall that in general, the duality theorem can only hold with a relaxed notion of integrability; cf.\ Examples~\ref{ex:noIntegrability} and~\ref{ex:gapWithIntegrability}. We have the following sufficient condition for integrability in the classical sense.

\begin{remark}\label{rk:integraleDualSuffGLlobal}
  Suppose that for each $k\geq1$, either $\mu_{k}$ is supported on a compact subset of $I_{k}$ or
  $$
	  u_{\nu_{k}} - u_{\delta_{m_{k}}} \leq C_{k}(u_{\nu_{k}} - u_{\mu_{k}})   
  $$
  for some constant $C_{k}$, where $m_{k}$ is the barycenter of $\mu_{k}$. Then, 
  $$
    \cD^{c}_{\mu_{k},\nu_{k}}(f)=\cD^{1}_{\mu_{k},\nu_{k}}(f),\quad k\geq1
  $$
  and in particular the optimizer in Theorem~\ref{th:dualityGlobal} satisfies $\varphi\in L^{1}(\mu)$ and $\psi\in L^{1}(\nu)$. Indeed, Remark~\ref{rk:finiteIntegralSuffCond} shows that all concave moderators can be chosen as $\chi=0$ in this situation.
\end{remark}

\begin{remark}\label{rk:dualityGlobal}
  In the setting of Theorem~\ref{th:dualityGlobal} and the notation of Proposition~\ref{pr:decomp} and Lemma~\ref{le:reductionToComponentsDual}, the following relations hold.
  \begin{enumerate}
  \item
    We have $\bS_{\mu,\nu}(f) = \sum_{k\geq0} \bS_{\mu_{k},\nu_{k}}(f)$ and $\bI_{\mu,\nu}(f) = \sum_{k\geq0} \bI_{\mu_{k},\nu_{k}}(f)$.
  \item
    If $P_{k}\in\cM(\mu_{k},\nu_{k})$ is optimal for $\bS_{\mu_{k},\nu_{k}}(f)$ for all $k\geq0$, then $P\in\cM(\mu,\nu)$ is optimal for $\bS_{\mu,\nu}(f)$. If $\bS_{\mu,\nu}(f)<\infty$, the converse holds as well: if $P\in\cM(\mu,\nu)$ is optimal for $\bS_{\mu,\nu}(f)$, then  $P_{k}\in\cM(\mu_{k},\nu_{k})$ is optimal for $\bS_{\mu_{k},\nu_{k}}(f)$ for all $k\geq0$.
  
  \item
  If $(\varphi_{k},\psi_{k},h_{k})\in \cD^{c}_{\mu_{k},\nu_{k}}(f)$ is optimal for $\bI_{\mu_{k},\nu_{k}}(f)$ for all $k\geq0$, then $(\varphi,\psi,h)\in \cD^{c}_{\mu,\nu}(f)$ is optimal for $\bI_{\mu,\nu}(f)$. If $\bI_{\mu,\nu}(f)<\infty$, the converse holds as well.
  \end{enumerate}
\end{remark}

\subsection{Monotonicity Principle}

An important consequence of the duality is the subsequent monotonicity principle describing the support of optimal transports; its second part can be seen as a substitute for the cyclical monotonicity from classical transport theory. While similar results have been obtained in \cite[Lemma~1.11]{BeiglbockJuillet.12} and \cite[Theorem~3.6]{Zaev.14}, the present version is stronger in several ways. First, it is stated with a set $\Gamma$ that is universal; i.e., independent of the measure under consideration; second, we remove growth and integrability conditions on~$f$; and third, the reward function is measurable rather than continuous.

\begin{corollary}[Monotonicity Principle]\label{co:monotonicityPrinciple}
  Let $f: \R^{2}\to [0,\infty]$ be Borel, let $\mu\leq_{c}\nu$ be probability measures and suppose that $\bS_{\mu,\nu}(f)<\infty$. There exists a Borel set $\Gamma\subseteq \R^{2}$ with the following properties.
  \begin{enumerate}
  \item A measure $P\in\cM(\mu,\nu)$ is concentrated on $\Gamma$ if and only if it is optimal for $\bS_{\mu,\nu}(f)$.
  
  \item Let $\bar\mu\leq_{c}\bar\nu$ be  probabilities on $\R$. If $\bar P\in\cM(\bar\mu,\bar\nu)$ is concentrated on $\Gamma$, then $\bar P$ is optimal for $\bS_{\bar\mu, \bar\nu}(f)$. %
  \end{enumerate}
  
  If $(\varphi,\psi,h)\in \cD^{c}_{\mu,\nu}(f)$ is a suitable\footnote{chosen as in Lemma~\ref{le:reductionToComponentsDual}\,(ii)} version of the optimizer from Theorem~\ref{th:dualityGlobal}, then we can take the following set for $\Gamma$,
  $$
   \big\{(x,y)\in \R^{2}:\, \varphi(x)+\psi(y)+h(x)(y-x) = f(x,y)\big\} \cap \bigg(\Delta \cup \bigcup_{k\geq1} I_{k}\times J_{k}\bigg).
  $$
\end{corollary}

\begin{proof}
  As $\bI_{\mu,\nu}(f)=\bS_{\mu,\nu}(f)<\infty$, Theorem~\ref{th:dualityGlobal} yields a dual optimizer $(\varphi,\psi,h)\in \cD^{c}_{\mu,\nu}(f)$ and we can define $\Gamma$ as above. 
  By Remark~\ref{rk:integralOfSuperhedge}, 
  \begin{equation}\label{eq:monPrincipleProof}
    P'(f)\leq P'[\varphi(X) + \psi(Y) + h(X)(Y-X)] = \mu(\varphi)+\nu(\psi)
  \end{equation}
  for all $P'\in\cM(\mu,\nu)$, whereas for $P\in\cM(\mu,\nu)$ with $P(\Gamma)=1$, the same holds with equality. This shows that $P(f)=\bS_{\mu,\nu}(f)$. For the converse in~(i), we observe that the inequality in~\eqref{eq:monPrincipleProof} is strict if $P'(\Gamma)<1$, and then $\bS_{\mu,\nu}(f) = \mu(\varphi)+\nu(\psi)$ shows that~$P'$ cannot be a maximizer.
  
  For the proof of~(ii), we choose a version of $(\varphi,\psi,h)\in \cD^{c}_{\mu,\nu}(f)$ as in Lemma~\ref{le:reductionToComponentsDual}\,(ii); moreover, we may assume that $\bar P(f)<\infty$. We shall show that $(\varphi,\psi,h)\in \cD^{c}_{\bar\mu,\bar\nu}(f)$; once this is established, the proof of optimality is the same as above.
  
  (a) On the one hand, we need to show that
  \begin{equation}\label{eq:monPrinciplePolarSets}
    \varphi(X) + \psi(Y) + h(X)(Y-X) \geq f(X,Y) \quad \cM(\bar\mu,\bar\nu)\qs
  \end{equation}   
  For this, it suffices to prove that the domains of the irreducible components of $\bar\mu\leq_{c}\bar\nu$ are subsets of the ones of $\mu\leq_{c}\nu$; i.e., that $u_{\mu}(x)=u_{\nu}(x)$ implies $u_{\bar\mu}(x)=u_{\bar\nu}(x)$, for any $x\in\R$. Indeed, let $u_{\mu}(x)=u_{\nu}(x)$. Since $\bar P$ is concentrated on $\Gamma \subseteq \Delta \cup \bigcup_{k\geq1} I_{k}\times J_{k}$, we know that $Y\geq x$ $\bar P$-a.s.\ on the set $\{X\geq x\}$. Writing $E[\,\cdot\,]$ for the expectation under $\bar P$, it follows that
  $$
    E[|X-x|\1_{X\geq x}] = E[(X-x)\1_{X\geq x}]=E[(Y-x)\1_{X\geq x}]= E[|Y-x|\1_{X\geq x}],
  $$
  where we have used that $E[Y|X]=X$ $\bar P$-a.s. An analogous identity holds for $\{X\leq x\}$, and thus
  $$
    u_{\bar\mu}(x)= E[|X-x|] = E[|Y-x|]=u_{\bar\nu}(x)
  $$
  as desired.
  
  (b) On the other hand, we need to show that $(\varphi,\psi)\in L^{c}(\bar\mu,\bar\nu)$. 
  By reducing to the components, we may assume without loss of generality that~$(\bar\mu,\bar\nu)$ is irreducible with domain $(I,J)$. As $(\varphi,\psi,h)\in \cD^{c,pw}_{\mu,\nu}(f)$ and $\bar P(\Gamma)=1$, we have
  $
    \bar P[ \varphi(X) +  \psi(Y)  + h(X)(Y-X)] = \bar P (f)<\infty,
  $
  and now Remark~\ref{rk:reverseIntegrability} implies that $(\varphi,\psi)\in L^{c}(\bar\mu,\bar\nu)$ as desired.
\end{proof}

We note that the dual optimizer $(\varphi,\psi,h)$ need not be unique, and a different choice may lead to a different set $\Gamma$. Moreover, we observe that an optimal $P\in\cM(\mu,\nu)$ need not exist. However, the following yields a fairly general sufficient criterion in the spirit of \cite{BeiglbockPratelli.12}.

\begin{remark}\label{rk:primalExistence}
  Let $f: \R^{2}\to [0,\infty]$ be Borel, let $\mu\leq_{c}\nu$ be probability measures and suppose that $\bS_{\mu,\nu}(f)<\infty$. Suppose there exist a Polish topology $\tau$ on $\R$ and a function $\bar{f}: \R^{2}\to [0,\infty]$ such that $\bar{f}$ is upper semicontinuous for $\tau\otimes\tau$ and $\bar f = f$ $\cM(\mu,\nu)$-q.s. Then, there exists an optimal $P\in\cM(\mu,\nu)$ for $\bS_{\mu,\nu}(f)$.
  
  Indeed, the induced weak topology on $\cM(\mu,\nu)$ does not depend on the choice of $\tau$; cf.\ \cite[Lemma~2.3]{BeiglbockPratelli.12}. Thus, under the stated conditions, the mapping $P\mapsto P(f)$ is upper semicontinuous on the compact set $\cM(\mu,\nu)$, and the result follows. We remark that compactness need not hold if non-product topologies are considered on $\R^{2}$, hence the use of $\tau\otimes\tau$. 
  
  The flexibility of choosing $\tau$ allows us to include a broad class of functions. Consider for instance $f$ of the product form $f(x,y)=f_{1}(x)f_{2}(y)$, where $f_{1}$ and $f_{2}$ are Borel measurable, or more generally any continuous function of $f_{1}(x)$ and $f_{2}(y)$. Then, we can choose $\tau$ such as to make $f$ continuous (cf.\ the proof of \cite[Theorem~1]{BeiglbockPratelli.12}) and the above applies.
\end{remark}

\begin{remark}\label{rk:KantorovichPotential}
   Corollary~\ref{co:monotonicityPrinciple} is a version of the classical ``Fundamental Theorem of Optimal Transport,'' see e.g.\ \cite[Theorem 2.13]{AmbrosioGigli.13}, where $\Gamma$ is the graph of the $c$-super\-differential of a $c$-concave function, the so-called Kantorovich potential (here $c=-f$ is the cost function). In our context, the roles of $\varphi$ and $\psi$ are not symmetric, and it is $\psi$ that constitutes the analogue of the Kantorovich potential. Indeed, $\varphi$ and $h$ can easily be obtained from $\psi$ by taking a concave envelope and its derivative, respectively; see the end of the proof of Proposition~\ref{pr:closednessIrred}.
\end{remark}

\section{Counterexamples}\label{se:counterexamples}

In this section, we present five counterexamples. Examples~\ref{ex:dualityGap} and~\ref{ex:noAttainment} show that the duality theory fails in the pointwise formulation; i.e.,
$$
  \varphi(x) + \psi(y) + h(x)(y-x) \geq f(x,y) \quad \mbox{for \emph{all}}\quad (x,y)\in \R^{2},
$$
and thus justify our quasi-sure approach. The subsequent two examples demonstrate that a relaxed notion of integrability is necessary for the dual elements, and the final example shows that duality fails if $f$ does not have any lower bound.

Our first example shows that a duality gap may occur with the pointwise formulation of the dual problem.

\pagebreak[4]

\begin{example}[Duality Gap in Pointwise Formulation]\label{ex:dualityGap}
  We exhibit a situation where
  \begin{enumerate}
   \item the reward function $f$ is bounded;
   \item a primal optimizer exists;
   \item if the dual problem is formulated in the pointwise sense, dual optimizers exist but there is a duality gap.
  \end{enumerate}
  
  Indeed, let $\mu$ be the restriction of the Lebesgue measure $\lambda$ to $[0,1]$. Setting $\nu=\mu$, the set $\cM(\mu,\nu)$ has a unique element, the law $P_{0}$ of $x\mapsto (x,x)$ under $\mu$, which is nothing but the uniform distribution on the diagonal of the unit square $[0,1]^{2}$. Consider the bounded reward function 
  $f(x,y):= \1_{x\neq y}$ which is lower (but not upper) semicontinuous. Since $P_{0}$ is concentrated on the diagonal, the primal value of the problem is
  $$
    \sup_{P\in\cM(\mu,\nu)} E^{P}[f]=E^{P_{0}}[f]=0.
  $$
  Now let $\varphi,\psi,h$ be Borel functions such that
  $$
    \varphi(x) + \psi(y) + h(x)(y-x) \geq f(x,y) \quad \mbox{for all}\quad x,y\in [0,1];
  $$
  then in particular
  $$
    \varphi(x) + \psi(y) + h(x)(y-x) \geq 1 \quad \mbox{for all}\quad x\neq y\in [0,1];
  $$
  Let $\eps>0$. By Lusin's theorem, there exists a Borel set $A\subseteq [0,1]$ with $\lambda(A)>1-\eps$ such that the restriction $\psi|_{A}$ is continuous. Using another fact from measure theory \cite[Exercise 1.12.63, p.\,85]{Bogachev.07volI}, the set $A$ can be chosen to be perfect; i.e., every point in $A$ is a limit point of $A$. Now let $x\in A$ and let $x_{n}\in A$ be a sequence of distinct points such that $x_{n}\to x$. Then passing to the limit in
  $$
    \varphi(x) + \psi(x_{n}) + h(x)(x_{n}-x) \geq 1 
  $$
  yields that 
  $$
    \varphi(x) + \psi(x) \geq 1 \quad \mbox{for all}\quad x\in A.
  $$
  As $\eps>0$ was arbitrary, it follows that $\lambda\{x\in[0,1]:\, \varphi(x) + \psi(x) \geq 1\}=1$. In particular, $\mu(\varphi)+\nu(\psi)\geq1$. This bound is attained, for instance, by the triplet $\varphi=1$, $\psi=0$, $h=0$, so that the dual problem in the pointwise formulation admits an optimizer and has value $1$; in particular, there is a duality gap in the pointwise formulation.
\end{example}

The next example shows that in general, the pointwise formulation fails to admit a dual optimizer. Such an example was already presented in~\cite{BeiglbockHenryLaborderePenkner.11}, using marginals with infinitely many irreducible components. The subsequence example shows that existence may fail even with finitely many (two) components and in a reasonably generic setting.

\begin{example}[No Dual Attainment in the Pointwise Formulation]\label{ex:noAttainment}
  We describe a setting where
  \begin{enumerate}
   \item the reward function is continuous and the marginals are compactly supported (but not irreducible);
   \item there is no duality gap for either formulation of the dual problem;
   \item there is no optimizer for the pointwise formulation of the dual problem.
  \end{enumerate}
  
  We fix two measures $\mu\leq_{c}\nu$ supported on $(-1,1)$ such that there are two irreducible components with domains $I_{1}\times J_{1}=(-1,0)^{2}$ and $I_{2}\times J_{2}=(0,1)^{2}$. Moreover, we assume that the origin is in the (topological) supports of $\mu$ and $\nu$;
  for instance, $\mu$ and $\nu$ could both be equivalent to the Lebesgue measure on~$(-1,1)$, or they could be discrete with atoms accumulating at the origin. The reward function $f$ is any continuous function of linear growth such that 
  $$
     f = 0\quad \mbox{on}\quad (-1,0)^{2}  \cup (0,1)^{2} \quad \mbox{and}
  $$ 
  \begin{center}
    $f$~~is not $(\mu\times \nu)$-a.s.\ bounded from above by a linear function on $(-1,0)\times (0,1)$.
  \end{center}
  An example is $f(x,y)=\sqrt{|xy|}\1_{(-1,0)\times (0,1)}$.
  
  Suppose for contradiction that $(\varphi,\psi,h)$ is a dual minimizer for the pointwise formulation; then
  $$
    \varphi(x)+\psi(y) + h(x)(y-x) \geq0,\quad (x,y)\in (-1,0)^{2}  \cup (0,1)^{2}.
  $$
  We have $\bS_{\mu,\nu}(f)=0$ and as $f$ is continuous with linear growth, there is no duality gap (even for the pointwise formulation); cf.\ \cite[Corollary~1.1]{BeiglbockHenryLaborderePenkner.11}. It follows that $P[\varphi(X)+\psi(Y) + h(X)(Y-X)]=0$ for all $P\in\cM(\mu,\nu)$ and thus
  $$
    \varphi(X)+\psi(Y) + h(X)(Y-X) =0\quad \cM(\mu,\nu)\qs
  $$  
  Let $N_{\mu}$ and $N_{\nu}$ be the corresponding nullsets as in Theorem~\ref{th:polarDescr} and write $I_{\mu}$ for $I\setminus N_{\mu}$ whenever $I$ is an interval. Then
  $$
    \varphi(x)+\psi(y) + h(x)(y-x) =0,\quad \! (x,y)\in [(-1,0)_{\mu}\times (-1,0)_{\nu}]\, \cup\, [(0,1)_{\mu}\times (0,1)_{\nu}]
  $$
  and in particular, fixing an arbitrary $x_{0}\in (0,1)_{\mu}$ yields
  $$
    \psi(y) = - \varphi(x_{0}) - h(x_{0})(y-x_{0}), \quad y\in (0,1)_{\nu},
  $$
  so that $\psi$ must be an affine function $\psi(y)=a_{+}y+d_{+}$ on $(0,1)_{\nu}$. It then follows that $h=-a_{+}$ on $(0,1)_{\mu}$ and $\varphi(x)=-a_{+}x-d_{+}$ on $(0,1)_{\mu}$, and a similar argument gives rise to constants $a_{-}, d_{-}$ for $(-1,0)$. Now, spelling out the condition
  $$
    \varphi(x)+\psi(y) + h(x)(y-x) \geq f(x,y)
  $$
  yields
  $$
  (a_{-}-a_{+}) y + (d_{-}-d_{+}) \geq f(x,y),\quad (x,y)\in (0,1)_{\mu} \times (-1,0)_{\nu},
  $$
  $$
  (a_{+}-a_{-}) y + (d_{+}-d_{-}) \geq f(x,y),\quad (x,y)\in (-1,0)_{\mu} \times (0,1)_{\nu}.
  $$
  Since $f(0,0)=0$ and $0$ is an accumulation point of  the intervals appearing on the right-hand side, it follows that $d_{-}=d_{+}$, but then it follows that $f$ is $(\mu\times \nu)$-a.s.\ bounded from above by a linear function on $(-1,0)_{\mu} \times (0,1)_{\nu}$, and all the same for $(0,1)_{\mu} \times (-1,0)_{\nu}$. This is the desired contradiction.
\end{example}

\begin{remark}\label{rk:nonExistenceGeneric}
  Nothing essential changes in Example~\ref{ex:noAttainment} if $\nu$ has one or more atoms at the boundary of the intervals $I_{k}$. As a matter of fact, the example suggests that one can expect non-existence for the pointwise formulation as soon as there are at least two adjacent irreducible components, the reward function is not Lipschitz where they touch, and the marginals exhibit some richness (in particular, have infinite support). 
\end{remark}

The next two examples concern the quasi-sure version of the dual problem; i.e., the setting of the main part of the present paper, and in particular the notion of integral introduced in Section~\ref{se:generalizedIntegral}. The first one shows that it is necessary to relax the notion of integrability in order to have existence for the dual problem $\bI_{\mu,\nu}$.

\begin{example}[Failure of Integrability for Optimizers]\label{ex:noIntegrability}
  We exhibit a situation where
  \begin{enumerate}
  \item the reward function $f$ is bounded;
  \item primal and dual optimizers exist and there is no duality gap;
  \item whenever $(\varphi,\psi,h)\in\cD^{c}_{\mu,\nu}(f)$ is a dual optimizer, $\varphi$ is not $\mu$-integrable and $\psi$ is not $\nu$-integrable. 
  \end{enumerate}
 
  Indeed, let $(c_{i})_{i\geq1}$ be a sequence of strictly positive numbers satisfying $\sum_{i}c_{i}=1$  such that the probability measure
  $$
    \mu := \sum_{i\geq1} c_{i} \delta_{i}
  $$
  has finite first moment but infinite second moment. Moreover, set
  $$
    \nu := \frac13 \sum_{i\geq1} c_{i} (\delta_{i-1} + \delta_{i} + \delta_{i+1})
  $$
  and note that the moments of $\nu$ then have the same property. Finally, our reward function is given by
  $
    f(x,y)= \1_{x\neq y}.
  $
  
  We observe that 
  $$
     P :=  \sum_{i\geq1} c_{i}\, \delta_{i} \otimes \frac13(\delta_{i-1} + \delta_{i} + \delta_{i+1}) \,\in \cM(\mu,\nu);
  $$
  in particular, $\mu\leq_{c}\nu$.  Moreover, let $\varphi(x)=-x^{2}$, $\psi(y)=y^{2}$ and $h(x)=-2x$; then we have
  $$
    \varphi(x) + \psi(y) + h(x)(y-x) = -x^{2} + y^{2} - 2x(y-x) = (x-y)^{2}\geq f(x,y)
  $$
  for all $(x,y)\in \N\times\N_{0}$, with equality holding on the set 
  $$
    \Gamma := \big\{(x,y)\in \N\times\N_{0}:\, y\in \{x-1,x,x+1\}\big\}.
  $$
  Since $P$ is concentrated on $\Gamma$, it follows as in Corollary~\ref{co:monotonicityPrinciple} that $P\in\cM(\mu,\nu)$ is a primal optimizer and $(\varphi,\psi,h)\in \cD^{c}_{\mu,\nu}(f)$ is a dual optimizer. One can observe that a concave moderator is given by $\chi(y)=-y^{2}$. 
  
  Now let $(\varphi,\psi,h)\in \cD^{c}_{\mu,\nu}(f)$ be an arbitrary optimizer; then we must have
  $$
    \varphi(x) + \psi(y) + h(x)(y-x) = f(x,y) \quad P\as
  $$
  and hence, by the definition of $P$, this equality holds for all $(x,y)\in\Gamma$. It follows that
  $$
 \mbox{for all }x\in\N,\quad\begin{cases}
    \varphi(x) + \psi(x-1) - h(x) =1, \\
    \varphi(x) + \psi(x+1) + h(x) =1, \\
    \varphi(x) + \psi(x) = 0.
  \end{cases}
  $$
  In particular, $\varphi = - \psi$ on $\N$ and
  $$
    2\varphi(x) -\varphi(x-1) - \varphi(x+1) =2, \quad x\in\N.
  $$
  All solutions of this difference equation satisfy 
  $$
    \varphi(x)=-x^{2}+bx+c, \quad x\in\N,
  $$
  for some constants $b,c\in\R$. In particular, $\varphi^{-}$ is not $\mu$-integrable and $\psi^{+}$ is not $\nu$-integrable, and as a result, there exists no optimizer for $\bI_{\mu,\nu}(f)$ in the class $\cD^{1}_{\mu,\nu}(f)\subset\cD^{c}_{\mu,\nu}(f)$.
\end{example}

The next example shows that without a relaxed notion of integrability, the dual problem may be infinite even if the primal problem $\bS_{\mu,\nu}$ is finite.

\begin{example}[Integrability Requirement Causes Duality Gap]\label{ex:gapWithIntegrability}
  We exhibit a situation where
  \begin{enumerate}
  \item the reward function $f$ is continuous;
  \item primal and dual problem are finite;
  \item the set $\cD^{1}_{\mu,\nu}(f)$ is empty; in particular, there is a duality gap if $\cD^{c}_{\mu,\nu}(f)$ is replaced by $\cD^{1}_{\mu,\nu}(f)$ in the definition of the dual problem $\bI_{\mu,\nu}$.
  \end{enumerate}
  
  Let $\mu\leq_{c}\nu$ be as in Example~\ref{ex:noIntegrability}; we now make the specific choice
  $$
   c_{i} = i^{-3}C,\quad i\in\N,
  $$
  where $C$ is the normalizing constant.
  This ensures that $\mu$ and $\nu$ have a first but no second moment. Moreover, the strict concavity of $i\mapsto i^{-3}$ implies that 
  $$
    \mu(\{i\}) > \nu(\{i\}), \quad i\in\N.
  $$
  The associated potential functions satisfy $u_{\mu}=u_{\nu}$ on $(-\infty,0]$. If there were $x>0$ with $u_{\mu}(x)=u_{\nu}(x)$, then as $\mu$ is the second (distributional) derivative of $u_{\mu}/2$,  we would have $\nu(\{x\})>\mu(\{x\})$, a contradiction. As a result,
  \begin{center}
  $\mu\leq_{c}\nu$ is irreducible with domain $(I,J)$ given by $I=(0,\infty)$, $J=[0,\infty)$.
  \end{center}
  For the reward function, we now consider 
  $$
    f(x,y)=(x-y)^{2}.
  $$
  As seen in in Example~\ref{ex:noIntegrability}, setting $\varphi(x)=-x^{2}$, $\psi(y)=y^{2}$ and $h(x)=-2x$  yields $(\varphi,\psi,h)\in \cD^{c}_{\mu,\nu}(f)$ with concave moderator $\chi(y)=-y^{2}$; in fact, $ \mu(\varphi)+\nu(\psi) = P(\1_{x\neq y})\leq 1$ in the notation of Example~\ref{ex:noIntegrability}, and thus $\bS_{\mu,\nu}(f)\leq1$. 
  
  Suppose that there exists some $(\varphi,\psi,h)\in \cD^{1}_{\mu,\nu}(f)$. Since $\mu\leq_{c}\nu$ is irreducible, Corollary~\ref{co:polarSetsIrred} shows that every point in $\N\times\N_{0}$ is charged by some element of $\cM(\mu,\nu)$ and hence
  $$
  \varphi(x) + \psi(y) +h(x)(y-x) \geq f(x,y)=x^{2}+y^{2} -2xy \quad\mbox{for all}\quad (x,y)\in \N\times\N_{0}.
  $$
  We see that $\psi$ must have at least quadratic growth in $y$, and thus $\psi\notin L^{1}(\nu)$ and $\varphi\notin L^{1}(\nu)$. As a result, $\cD^{1}_{\mu,\nu}(f)=\emptyset$ and the corresponding dual problem has infinite value, whereas the primal one satisfies $0\leq \bS_{\mu,\nu}(f)\leq1$.
\end{example}

Our last example shows that a duality gap may occur (even in the quasi-sure formulation) if $f$ does not have any lower bound. This should be compared with \cite[Theorem~1]{BeiglbockHenryLaborderePenkner.11} which shows that there is no duality gap if $f$ is upper semicontinuous with values in $[-\infty,\infty)$.

\begin{example}[Duality Gap Without Lower Bound]\label{ex:dualityGapLower}
  We exhibit a situation where
  \begin{enumerate}
   \item the reward function $f$ takes values in $[-\infty,0]$;
   \item primal and dual optimizers exist;
   \item there is a duality gap.
  \end{enumerate}
  
  Indeed, let $\mu=\lambda|_{[0,1]}$ be the restriction of the Lebesgue measure to $[0,1]$, fix a constant $\Delta>0$ and 
  $$
    \nu = \frac{1}{2} \Big( \lambda|_{[-\Delta,1-\Delta]}+ \lambda|_{[\Delta,1+\Delta]}\Big).
  $$
  Then $\mu\leq_{c}\nu$ is irreducible with domain given by $I=J=(-\Delta,1+\Delta)$. Indeed, a particular element of $\cM(\mu,\nu)$ is given by $P=\mu\otimes\kappa$, where
  $$
  \kappa(x) = \frac{1}{2} \big( \delta_{x-\Delta}+ \delta_{x+\Delta}\big).
  $$
  For the reward function, we choose
  $$
    f(x,y)=\begin{cases}
    0 & \mbox{if } |x-y|<\Delta,\\
    -1 & \mbox{if } |x-y|=\Delta,\\
    -\infty & \mbox{if } |x-y|>\Delta.
    \end{cases}
  $$
  
  We first analyze the primal problem. Let $P'\in\cM(\mu,\nu)$ and let 
  $P'=\mu\otimes \kappa'$ be a disintegration. We observe that
  $$
    \int (x-y)^{2}\, \kappa'(x,dy)\, \mu(dx) = \Var(\nu) - \Var(\mu) = \Delta^{2}. 
  $$ 
  If $P'(f)>-\infty$, then $\kappa'(x)\{|x-y|>\Delta\}=0$ for $\mu$-a.e.\ $x$ and the above implies that $|x-Y|=\Delta$ for $\mu$-a.e.\ $x$ and therefore $P'=P$. 
  As a result, $P'(f)=-\infty$ for all $P\neq P'\in\cM(\mu,\nu)$ and
  $$
    \sup_{P'\in\cM(\mu,\nu)} E^{P'}[f]=E^{P}[f]=-1.
  $$
  
  We now turn to the dual problem; since $\mu\leq_{c}\nu$ is irreducible, the quasi-sure formulation is equivalent to the pointwise one. Let $\varphi,\psi,h$ be Borel functions such that
  $$
    \varphi(x) + \psi(y) + h(x)(y-x) \geq f(x,y) \quad \mbox{for all}\quad (x,y)\in I\times J;
  $$
  then in particular
  \begin{align*}
    \varphi(x) + \psi(x+\delta) + h(x)\delta \geq 0 \quad \mbox{for all}\quad x\in(0,1),\quad \delta\in [0,\Delta),\\
    \varphi(x) + \psi(x-\delta) - h(x)\delta \geq 0 \quad \mbox{for all}\quad x\in(0,1),\quad \delta\in [0,\Delta).
  \end{align*}
  Adding these two inequalities yields
  $$
    \varphi(x) + \frac{\psi(x-\delta)+\psi(x+\delta)}{2} \geq 0 \quad \mbox{for all}\quad x\in(0,1),\quad \delta\in [0,\Delta).
  $$
  Let $\eps>0$. As in Example~\ref{ex:dualityGap}, Lusin's theorem can be used to find a set $A\subseteq (0,1)$ with $\lambda (A)>1-\eps$ such that for all $x\in A$ there exists a sequence $\delta_{n}=\delta_{n}(x)$ with $\psi(x\pm\delta_{n})\to \psi(x\pm \Delta)$. Thus, passing to the limit in the above inequality shows that 
  $$
    \varphi(x) + \frac{\psi(x-\Delta)+\psi(x+\Delta)}{2} \geq 0 \quad \mbox{for all}\quad x\in A,
  $$
  and as $\eps>0$ was arbitrary, the inequality holds $\mu$-a.e.
  But then
  $$
    \mu(\varphi)+\nu(\psi) = P[\varphi(X)+\psi(Y)] = \int \varphi(x) + \frac{\psi(x-\Delta)+\psi(x+\Delta)}{2}\,\mu(dx)\geq0. 
  $$
  As a result, the dual value is zero and a dual optimizer is given for instance by $\varphi=\psi=h=0$.
\end{example}

\newcommand{\dummy}[1]{}

\end{document}